\date{October 27, 2014}

\documentclass[12pt]{amsart}
\usepackage{latexsym,amsmath,amsfonts,amscd,amssymb}
\usepackage{graphics}
\newtheorem{theorem}{Theorem}[section]
\newtheorem{lemma}[theorem]{Lemma}
\newtheorem{proposition}[theorem]{Proposition}
\newtheorem{corollary}[theorem]{Corollary}

\newtheorem{definition}[theorem]{Definition}

\newtheorem{example}[theorem]{Example}

\theoremstyle{remark}
\newtheorem{remark}[theorem]{Remark}

\newcommand{\la}{\langle}
\newcommand{\ra}{\rangle}

\newcommand{\Arg}{\text{Arg}}

\newcommand{\unit}{\textbf{1}}

\newcommand{\Div}{\operatorname{div}}

\newcommand{\cD}{{\mathcal D}}

\newcommand{\cL}{{\mathcal L}}

\newcommand{\cP}{{\mathcal P}}

\newcommand{\cS}{{\mathcal S}}

\newcommand{\CC}{{\mathbb C}}

\newcommand{\NN}{{\mathbb N}}

\newcommand{\QQ}{{\mathbb Q}}
\newcommand{\RR}{{\mathbb R}}

\newcommand{\ZZ}{{\mathbb Z}}

\newcommand{\bk}{{\mathbf{k}}}

\title[Explicit and trace formulas via Poisson-Newton 
formula]{Unified treatment of explicit and trace formulas via Poisson-Newton formula}


\author[V. Mu\~{n}oz]{Vicente Mu\~{n}oz}
\address{Facultad de
Matem\'aticas, Universidad Complutense de Madrid, Plaza de Ciencias
3, 28040 Madrid, Spain}
\email{vicente.munoz@mat.ucm.es}

\author[R. P\'{e}rez Marco]{Ricardo P\'{e}rez Marco}
\address{CNRS, LAGA UMR 7539, Universit\'e Paris XIII,
99, Avenue J.-B. Cl\'ement, 93430-Villetaneuse, France}
\email{ricardo@math.univ-paris13.fr}

\begin{document}

\maketitle

\begin{abstract}
  We prove that a Poisson-Newton formula, in a broad sense, is associated to each Dirichlet
series with a meromorphic extension to the whole
complex plane of finite order. These formulas simultaneously generalize the classical Poisson formula and Newton
formulas for Newton sums. Classical Poisson formulas in Fourier analysis,
explicit formulas in number theory and Selberg trace formulas in Riemannian
geometry appear as special cases of our general Poisson-Newton formula. 
\end{abstract}

\noindent \emph{We dedicate this article to Daniel Barsky and Pierre Cartier
 for their interest and constant support.}

\section{Introduction} \label{sec:intro}

All classical Poisson formulas for functions in Fourier analysis result from the general distributional Poisson formula
\begin{equation}\label{eqn:a1}
\sum_{n\in \ZZ} e^{i\frac{2\pi}{\lambda} n t}  =\lambda \sum_{k\in \ZZ}  \delta_{\lambda k} \, ,
\end{equation}
which is an identity of distributions identifying an infinite sum of
exponentials, converging in the sense
of distributions, and a purely atomic distribution. This distributional formula
is related to the simplest finite
Dirichlet series
$$
f(s)=1-e^{-\lambda s} \ .
$$
It is interesting to observe that on the left hand side of (\ref{eqn:a1})
we have an exponential sum 
$$
W(f)=\sum_\rho e^{\rho t} \ ,
$$
where the sum runs over the zeros $\rho_n = \frac{2\pi i}{\lambda} n$, $n\in \ZZ$, of $f$,
and on the right hand side of (\ref{eqn:a1})
we have a sum of atomic masses at the multiples of the fundamental frequency $\lambda$.
One can say that the frequencies associated to the zeros are resonant at the fundamental frequencies. Taking
the Fourier transform we obtain the dual Poisson formula that is of the same form where we exchange zeros
and fundamental frequencies. Thus the fundamental frequencies are also resonant at the zeros.

The main purpose of this article is to show that this type of formulae are 
general and to each meromorphic Dirichlet series $f$
we can associate a
distributional Poisson formula
\begin{equation}\label{eqn:a2}
W(f)=\sum_\rho n_\rho e^{\rho t} =\sum_{\bk} \langle \boldsymbol{\lambda} , \bk\rangle  b_\bk \,\delta_{\langle \boldsymbol{\lambda} , \bk\rangle} \, ,
\end{equation}
where the first sum of exponentials runs over the divisor of $f$, i.e., zeros and poles $\rho$ with multiplicities $n_\rho$, and the second sum runs over non-zero sequences $\bk=(k_1, k_2,\ldots )\in
\NN^\infty$ of non-negative integers, all of them zero but finitely many,
and $\langle \boldsymbol {\lambda} , \bk\rangle = \sum \lambda_j k_j$. The coefficients $b_\bk$ are determined by 
the formula $- \log f(s)= \sum_{\bk} b_\bk e^{-\langle\boldsymbol {\lambda} , \bk\rangle s}$. The
equality holds in $\RR_+^*$. Conversely, we prove that
any such Poisson formula comes from a Dirichlet series.

The distribution
$$
W(f)=\sum_\rho n_\rho e^{\rho t}
$$
is well defined in $\RR_+^*$ and is called the \emph{Newton-Cramer distribution} of $f$.
We name it after Newton because it appears as a distributional
interpolation of the Newton sums to exponents $t\in \RR$, since in the complex
variable\footnote{The variable $z=e^s$ or better $z=e^{-s}$ is the proper variable when
dealing with Dirichlet series.}
$z=e^{s}$ the zeros are the $\alpha =e^{\rho}$ so, for simple zeros such that $\rho-\rho'\not= 2\pi i k$, $k\in \ZZ$,
$$
W(f)(t)=\sum_\alpha \alpha^{ t} \, ,
$$
and for integer values $t=m\in \ZZ$ we recognize (in case of convergence) the Newton sums
 $$
W(f)(m)=S_m=\sum_\alpha \alpha^{m} \, .
$$
There is a precise theorem behind this observation. We show that
our Poisson-Newton formula for a finite Dirichlet series  $f$ with a single fundamental 
frequency is strictly equivalent to
the classical Newton relations. This is the reason why we name also after Newton our general Poisson formulas.

Writing $\rho=i\gamma$ we see that the sum $W(f)$ of the left hand side of (\ref{eqn:a2})
is the Fourier transform of the atomic
Dirac distributions $\delta_\gamma$ and we can formally write
$$
\sum_\gamma n_{i\gamma} \hat \delta_\gamma =\sum_{\bk} \langle \boldsymbol {\lambda} ,
\bk\rangle  b_\bk \,\delta_{\langle \boldsymbol {\lambda} , \bk\rangle} \, .
$$
The form of this formula, relating zeros to fundamental frequencies,
strongly reminds other distributional formulas in other contexts.
In number theory, more precisely in the theory of zeta and $L$-functions,
the same type of identities do appear as ``explicit formulas''
associated to non-trivial zeros of the zeta and other $L$-functions. These \textit{explicit formulas}, when
written in distributional form, reduce to a single distributional relation that identifies a sum of
exponentials associated to the divisor of the zeta or $L$-function and an atomic distribution associated
to the location of prime numbers. Usually the sum runs over non-trivial zeros, and the sum over trivial zeros appears
hidden in other forms as a \textit{Weil functional}, which is classically interpreted as 
corresponding to the ``infinite prime''\footnote{It may be more appropriate to talk of the ``prime'' $p=1$.
Actually, Arakelov theory suggests that what is usually understood as ``prime'' at infinity is better 
understood as ``prime'' 1 (cf.\ the ``field'' with one element).}.
For that reason, Delsarte labeled this formula as ``Poisson formula
with remainder'' (see \cite{D}), the ``remainder'' refers to the sum over the trivial part of the divisor. More precisely, for
the Riemann zeta function, we have in $\RR_+^*$ (see \cite{Lang})
$$
\sum_\rho n_\rho e^{\rho t} +W_0(f) =-\sum_p \sum_{k\geq 1}  \log p \, \delta_{k\log p} \ ,
$$
where the sum on the left runs over the non-trivial (i.e., non-real) zeros $\rho$, and the sum over 
$p$ runs over prime numbers. Conjecturally, the non-trivial zeros are simple, i.e., $n_\rho=1$. The term $W_0(f)$ is the sum over the
trivial (real) divisor and is computable
$$
W_0(f)(t)=-e^t+\sum_{n\geq 1} e^{-2n t} =-e^t + \frac{1}{e^{2t} -1} \ ,
$$
and corresponds to Delsarte ``remainder'', or to the Weil functional of the infinite
prime. Also we have in this case
$$
\sum_\rho n_\rho e^{\rho t}=e^{t/2} V(t) + e^{t/2} V(-t) \ ,
$$
where
$$
V(t)=\sum_{\Re \gamma >0} e^{i\gamma t}  \ ,
$$
is the classical Cramer function, studied by H. Cramer \cite{C}, where $\rho=\frac12+i\gamma$. 
This motivates that
we name our distribution $W(f)$
also after Cramer.

\medskip

In Riemannian geometry, we have the same structure
for the Selberg trace formula for compact surfaces with constant negative curvature. With the relevant
difference that Selberg zeta function is of order $2$, which gives a ``remainder'' of order $2$ also.
Selberg formula relates the length
of primitive geodesics, which play the role of prime
numbers, and the eigenvalues of the Laplacian, which
give the zeros of the Selberg zeta function. 
For non-negative constant curvature, the formulas are of
a different nature and the distribution on the right side are no longer simple atomic Diracs, 
but also higher order derivatives appear \cite{DG}.
It is well known that one of Selberg's motivation was the analogy with the explicit formula in 
Number Theory, that our approach explains. According to B. Conrey \cite{conrey}

\textit{"The trace formula resembles the explicit formula in certain ways. Many researchers have attempted to interpret Weil's explicit formula in terms of Selberg's trace formula."}

In the context of dynamical systems and semiclassical quantization,
we have Gutzwiller trace formula (see \cite{GU}), which relates the periods of the periodic orbits
(frequencies of the zeta function)
of a classical mechanical system to the energy levels (zeros of the zeta function)
of the associated quantum system.

Our goal is to put in the proper context, generalize and
make precise the analogy of Poisson and trace formulas, and derive a general class
of Poisson formulas that contain all such instances. More precisely, to each meromorphic Dirichlet series of
finite order we associate a Poisson-Newton formula. All relevant known formulas can be generated in this way.
On the other hand the fact that \textit{explicit formulas} in number theory and Selberg trace formula can be
seen as a generalization of Newton formulas, seems to be a new interpretation.
It is important to remark that in our general setting the Poisson-Newton formulas are
independent from a possible functional equation for the
Dirichlet series $f$, contrary to what happens in classical formulas (see remark \ref{rem:functional-equation}
for the precise formulation).

\section{Newton-Cramer distribution}\label{sec:newton-cramer}

Let $f$ be a meromorphic function on the complex plane $s\in \CC$ of finite order 
(see \cite{A} for classical results
on meromorphic functions). 
We denote by $(\rho)$ the set of zeros and poles
of $f$, and the integer $n_\rho$ is the multiplicity of $\rho$ (positive for zeros and
negative for poles, with the convention $n_\rho =0$ if $\rho$ is neither a zero nor pole). 
The convergence exponent of $f$ is the minimum
integer $d\geq 0$ such that
 \begin{equation}\label{eqn:pesado}
\sum_{\rho \not= 0} |n_\rho| \, |\rho|^{-d} < +\infty\ .
  \end{equation}
We have $d=0$ if and only if $f$ is a rational function.
We shall assume henceforth that $d\geq 1$.
The order \, $o$ \,  of $f$ satisfies $ d \leq [o]+1$. 
We shall also assume that there is some $\sigma_1\in \RR$ such that $\Re \rho\leq \sigma_1$ for any
zero or pole $\rho$ of $f$.

Associated to the divisor $\Div(f)= \sum n_\rho\,\rho$, we define a
distribution 
 \begin{equation}\label{eqn:N-C0}
  W(f)= \sum n_\rho \, e^{\rho t}
  \end{equation}
on $\RR^*_+=]0,\infty[$, called the Newton-Cramer distribution. 
Moreover, we define a distribution on the whole of $\RR$, vanishing
on $\RR^*_-=]-\infty,0[$. This means that we have to make sense of the structure of $W(f)$ at $0\in \RR$.

We start by fixing the space of distributions we will be working on, which
are the distributions \emph{Laplace transformable} in the terminology of
\cite[Section 8]{Z}. We denote $\cD=C^\infty_0$ so $\cD'$ is the space of
all distributions. The space $\cS$ of $C^\infty$-functions of rapid decay on $\RR$ consists of those 
$\varphi$ such that for any $n,m>0$, $|t^n \frac{d^m}{dt^m} \varphi |\leq C_{nm}$.
The dual space $\cS'$ is the Schwartz space of tempered distributions.
Then we say that a distribution $W\in \cD'$ is (right) Laplace transformable if
its support is included in $]-M,\infty[$, for some $M\in\RR$, and there exists
some $c\in \RR$ such that $W\, e^{-ct}\in \cS'$. The topology is as follows: a
sequence $W_k$ of Laplace transformable distributions converge to some $W$ if
there exists some uniform $M,c$ such that $\text{Supp} (W_k) \subset ]-M,\infty[$
and $W_k \, e^{-ct} \to W\, e^{-ct}$ in $\cS'$. Note also that the space of 
Laplace transformable distributions is stable by differentiation, and differentiation 
is continuous in this space (and the same constant $c$ works for the sequence of derivatives).

To make sense  of (\ref{eqn:N-C0}), we are going to construct explicitly a $d$-th primitive of it
which is a continuous function. For this, we introduce a parameter 
$\sigma \in \CC$ which serves as origin for defining the primitive of $W(f)$. The relevant function is
$$
K_{d,\sigma}(t)=\left ( n_{\sigma} \frac{t^d}{d!}\right ) \,  \emph \unit_{\RR_+} +\sum_{\rho \not= \sigma} 
 \left( \frac{n_\rho}{(\rho-\sigma)^d} e^{(\rho-\sigma)t} \right) \emph \unit_{\RR_+} .
$$
This is a Laplace transformable distribution (see Lemma \ref{lem:1} below).
Set
\begin{equation}\label{eqn:N-C}
W_\sigma(f) = e^{\sigma t} \frac{D^d}{Dt^d} \left( (K_{d,\sigma}(t)-K_{d,\sigma}(0)) \emph \unit_{\RR_+} \right),
\end{equation}
where $\frac{D}{Dt}$ denotes the distributional derivative. The following justifies our definition of the 
Newton-Cramer distribution (\ref{eqn:N-C0}).

\begin{lemma} \label{lem:1}
 For finite sets $A$, consider the family of locally integrable functions
 $$
 \tilde{W}_A(f)=\left( \sum_{\rho \in A} n_\rho e^{\rho t}\right) {\emph\unit}_{\RR_+} \, .
 $$
 There is a family of Laplace transformable distributions $W_{A,\sigma}(f)$ which 
coincides with $\tilde{W}_A(f)$ in $\RR^*$, and which
 converges in $\RR$ (over the filter of finite sets $A$), to the Laplace transformable 
distribution $W_\sigma(f)$ (converging
as Laplace transformable distributions).
\end{lemma}

\begin{proof}
We prove first the lemma when $\sigma$ is not a zero nor pole of $f$. Let $\alpha=
\sigma_1-\Re \sigma \in \RR$.
We define
 \begin{equation} \label{eqn:Kl}
K_\ell(t)= 
\sum_{\rho } \left( \frac{n_\rho}{(\rho-\sigma)^\ell} e^{(\rho-\sigma)t} \right) \unit_{\RR_+} .
   \end{equation}
(We drop the subscript $\sigma$ from the notation during this proof.)
Then for $\ell \geq d$, $K_\ell$ is absolutely convergent for $t\in \RR_+$, since
$$
\left |  e^{(\rho-\sigma)t} \right | =e^{\Re (\rho-\sigma) t} \leq e^{ \alpha t} \ ,
$$
and
$$
|e^{- \alpha\,  t} K_\ell| \leq \sum_\rho \frac{|n_\rho|}{|\rho-\sigma|^\ell} <\infty \ . 
$$
So  $K_\ell$ is a uniformly convergent on compact subsets of $\RR_+$, and hence it is continuous in $\RR^*$.
The function
 $$
  F_{\ell}(t)=(K_{\ell}(t)-K_{\ell}(0) ) \unit_{\RR_+}
 $$
is a  continuous function on $\RR$, for $\ell \geq d$.

For a finite set $A$, denote by
 $$
 K_{\ell,A}(t)= \sum_{\rho\in A} \left(
 \frac{n_\rho}{(\rho-\sigma)^\ell} e^{(\rho-\sigma)t}\right) \unit_{\RR_+}
 $$
the corresponding sum over $\rho\in A$, and $F_{\ell ,A}(t)=(K_{\ell,A}(t)-K_{\ell,A}(0)) \unit_{\RR_+}$.
Then $F_{\ell,A}\to F_\ell$ uniformly on compact subsets of $\RR$. More precisely,
$|F_{\ell,A}-F_\ell| \leq c_A e^{\alpha t}$ with $c_A \to 0$. In particular,
$e^{-(\alpha+\epsilon) t}F_{\ell,A} \to e^{-(\alpha+\epsilon) t}F_\ell$ in $\cS'$, for some $\epsilon>0$.
On $\RR^*$,
 $$
 \tilde{W}_A(f)=
  \left(\sum_{\rho\in A} n_\rho \, e^{\rho t} \right)
  \unit_{\RR_+} =e^{\sigma t}\frac{d^{d}}{dt^d} F_{d,A}(t) .
$$
We consider
 $$
 {W}_{A,\sigma} (f)= e^{\sigma t}\frac{D^{d}}{Dt^d} F_{d,A} ,
$$
taking the distributional derivative. 

Let $K$ be a smooth function on $\RR$. It is easy to check that
$\frac{D}{Dt} (K \unit_{\RR_+} ) =K' \unit_{\RR_+}  + K(0) \delta_0$,
as distributions in $\cD'$. This formula holds also for Laplace transformable distributions. 
We apply this to $F_{\ell,A}$, using that $K'_{\ell , A}=K_{\ell -1, A}$. We get
\begin{align*}
 &\frac{D^{d}}{Dt^d} F_{d,A} =K_{0,A} (t)  + K_{1,A} (0) \delta_0 + K_{2,A} (0) \delta'_0 +\ldots +K_{d-1,A} (0) \delta_0^{(d-2)} \\ 
 &=K_{0,A} (t)  + \left (\sum_{\rho \in A} \frac{n_\rho}{\rho-\sigma}\right ) \delta_0 + 
\left (\sum_{\rho \in A} \frac{n_\rho}{(\rho-\sigma)^2}\right ) \delta'_0 + \ldots 
+ \left (\sum_{\rho \in A} \frac{n_\rho}{(\rho-\sigma)^{d-1}}\right ) \delta_0^{(d-2)} \ .
\end{align*}
Thus the difference between $\tilde{W}_A(f)$ and $W_{A,\sigma} (f)$ is a distribution supported
at $\{0\}$.

We have the convergence $F_{d,A}\to F_{d}$ as Laplace transformable distributions
(with $c=\alpha+\epsilon$ fixed uniformly).
Then differentiating, we have $W_{A,\sigma}(f)\to W_\sigma(f)$ as Laplace transformable
distributions, where
 $$
W_\sigma(f)= e^{\sigma t}\frac{D^{d}}{Dt^d} F_d\ ,
 $$
which is a Laplace transformable distribution with support on $\RR_+$, as stated.

When $\sigma$ is part of the divisor, then we do the same proof with
\begin{equation*} 
K_\ell(t)= \left ( n_{\sigma} \frac{t^\ell}{\ell !}\right ) \,   \unit_{\RR_+} +
\sum_{\rho \not= \sigma} \left( \frac{n_\rho}{(\rho-\sigma)^\ell} e^{(\rho-\sigma)t} \right) \unit_{\RR_+} ,
\end{equation*}
which adds to $W_\sigma (f)$ a term $n_{\sigma} e^{\sigma t}$.
\end{proof}

\begin{definition} 
 We call $W_\sigma(f)$ defined in (\ref{eqn:N-C}) the 
Newton-Cramer distribution associated to $f$ (with parameter $\sigma\in \CC$).
\end{definition}

We can write
$$
 W(f)=W_\sigma(f)|_{\RR_+^*} 
 = \lim_A \tilde{W}_A(f)|_{\RR_+^*} = \sum_{\rho} n_\rho\,e^{\rho t} \, ,
 $$
as a distribution on $\RR_+^*$. 
Note that $W(f)$ is independent of $\sigma$, since the only
dependence on $\sigma$ of $W_\sigma(f)$ is located at $0$.

\begin{proposition} \label{prop:1}
The distribution $W_{\sigma}(f)$ has support contained in $\RR_+$, and it
is Laplace transformable (with $c>\sigma_1$). 
%
\end{proposition}

\begin{proof}
By definition, $W_\sigma(f)$ is a Laplace transformable distribution. 
By the proof of Lemma \ref{lem:1}, we have that
$e^{-ct}\, F_d \in \cS'$ with $c=\alpha+\epsilon$, 
$\alpha=\sigma_1- \Re \sigma $, $\sigma_1=\sup \{\Re \rho\}$. 
As $W_\sigma(f)= e^{\sigma t} \frac{D^d}{Dt^d} F_d$,
we have that $e^{-(\Re \sigma)t} e^{-ct} \, W_\sigma(f) \in \cS'$.
So this means that 
 $$
 e^{-(\sigma_1+\epsilon)t} W_\sigma(f) \in \cS'
 $$
This means that 
we can pair $W_\sigma(f)$ with $e^{-st}$ for $\Re s>\sigma_1$,

$$
\langle W_\sigma(f) , e^{-st} \rangle := \langle  e^{-(\sigma_1+\epsilon)t} W_\sigma(f) ,
\lambda (t) e^{-(s-(\sigma_1+\epsilon))t}\rangle \ ,
$$
where $\lambda (t)$ is any infinitely smooth function with support bounded on the left which equals 
$1$ over a neighborhood of the support of $W_\sigma(f)$ (see \cite[p.\ 223]{Z}). 
%
\end{proof}

Let $f(s)$ be a meromorphic function with exponent of convergence $d$, 
and fix $\sigma$ as before. We have
defined a distribution $W_\sigma(f)(t)$ 
supported on $\RR_+$. If we make the change
of variables $t\mapsto -t$, we have the distribution $W_\sigma(f)(-t)$ 
defined as
 $$
 W_\sigma(f)(-t)=  
 (-1)^d e^{-\sigma t} \frac{D^d}{Dt^d} \left( (K_{d,\sigma}(-t)
 -K_{d,\sigma}(0)) \mathbf{1}_{\RR_-} \right).
 $$
 This is independent of $\sigma$ on $\RR^*_-$ and has a contribution at zero dependent on the parameter.
 
\begin{definition} 
We define the symmetric Newton-Cramer distribution
as  $\widehat{W}_\sigma(f)=W_\sigma(f)(t)+  W_\sigma(f)(-t)$.
\end{definition}

The symmetric Newton-Cramer distribution is a distribution supported on the
whole of $\RR$ and symmetric. It satisfies that there is some 
$c>0$ such that $h(t) \widehat{W}_\sigma(f) \in\cS'$, where
$h$ is smooth with $h=e^{-(c+\epsilon)|t|}$ for $|t|\geq 1$.
Note also that the only dependence on $\sigma$ is at zero.

Formally, $W_\sigma(f)(t)=\left( \sum n_\rho e^{\rho t}\right) \unit_{\RR_+}$  
and  $W_\sigma(f)(-t)=\left(\sum n_\rho e^{-\rho t}\right) \unit_{\RR_-}$, so we can write
 $$
 \widehat{W}_\sigma(f)=\sum_\rho n_\rho e^{\rho |t|}.
 $$

\section{Poisson-Newton formula}\label{sec:poisson-newton}

Let $f$ be a meromorphic function on $\CC$ of finite order. Let 
$(\rho)$ be the divisor defined by $f$, and assume that the convergence 
exponent is $d$, that is, (\ref{eqn:pesado}) is satisfied. 
We have the Hadamard factorization of $f$ (see \cite[p.\ 208]{A})
$$
f(s)=s^{n_0} e^{Q_f(s)} \prod_{\rho \not=0 } E_m (s/\rho )^{n_\rho} \ ,
$$
where $m=d-1\geq 0$ is minimal for the convergence of the product with
$$
E_m(z)=(1-z) e^{z+\frac12 z^2 +\ldots +\frac1m z^m} \ ,
$$
and $Q_f$ is a polynomial uniquely defined up to the addition of an integer multiple of $2\pi i$. The genus 
of $f$ is defined as the integer
$$
g=\max (\deg Q_f , m) \ ,
$$
and in general we have $d\leq g+1$ and 
$g\leq o \leq g+1$ (see \cite[p.\ 209]{A}). 

The origin plays no particular role, thus we may prefer 
to use the Hadamard product with origin at some $\sigma \in \CC$,
\begin{equation}\label{eqn:hadamar}
f(s)=(s-\sigma)^{n_\sigma} e^{Q_{f, \sigma} (s)} \prod_{\rho \not=\sigma } E_m 
\left (\frac{s-\sigma}{\rho -\sigma} \right )^{n_\rho} \ .
 \end{equation}
This can be obtained as follows: translate the divisor $(\rho)$ to $(\rho- \sigma)$,
consider the usual Hadamard factorization, and then do the change of 
variables $s\mapsto s-\sigma$. 

Taking the logarithmic derivative of (\ref{eqn:hadamar}), we obtain
 \begin{equation} \label{eqn:psi-psi2}
 \frac{f'(s)}{f(s)}= Q_{f, \sigma}' (s) + G(s),
 \end{equation}
where 
\begin{align*}
 G(s)&=\frac{n_{\sigma}}{s-\sigma}- \sum_{\rho \not=\sigma} n_\rho \left (
  \frac{1}{\rho-s} - \sum_{l=0}^{d-2} \frac{(s-\sigma)^l}{(\rho-\sigma)^{l+1}} \right ) \\
&=\frac{n_{\sigma}}{s-\sigma}+\sum_{\rho  \not=\sigma} n_\rho
  \frac{(s-\sigma)^{d-1}}{(\rho-\sigma)^{d-1}} \frac{1}{s-\rho} \, .
 \end{align*}
is a meromorphic function on $\CC$, which has a simple pole with residue $n_\rho$ at each $\rho$.

Note that $P_{f,\sigma} =-Q_{f,\sigma}'$ is a polynomial of degree $\leq g-1$. We call it
the \emph{discrepancy polynomial\/} of $f$. We have 
$$
P_{f,\sigma} =G-\frac{f'}{f}\, . 
$$

The main result of this section is the following Poisson-Newton formula for
a general meromorphic function $f$ of finite order and with divisor contained in a
left half plane, as above. Denote by $\cL$ the Laplace transform
and by $\cL^{-1}$ the inverse operator.

\begin{theorem} \label{thm:Poisson-Newton}
 Let $f$ be a meromorphic function of finite order with convergence exponent $d$ and
its divisor contained in a left half plane. Fix $\sigma\in \CC$. Let $W_\sigma(f)$
be its Newton-Cramer distribution and 
$P_{f,\sigma}(s)=c_0(\sigma)+c_1(\sigma) s+\ldots +c_{g-1}(\sigma) s^{g-1}$ 
be the discrepancy polynomial.
We have (as distributions on $\RR$),
$$
W_\sigma(f)=\sum_{j=0}^{g-1} c_j(\sigma) \delta_0^{(j)} + \cL^{-1} (f'/f) .
$$
\end{theorem}

\begin{proof}
We prove the theorem by taking the right-sided Laplace transform of $W_\sigma(f)$.
Here we have to choose $\Re s>\sigma_1$.
 \begin{align*}
  \cL (W_\sigma(f)) &= \la W_\sigma(f) , e^{-st} \ra_{\RR_+}
  = \left \la \frac{D^{d}}{Dt^d} F_{d}(t) , e^{(\sigma-s)t} \right \ra_{\RR_+} \\
 &=  \int_0^\infty (-1)^d (K_d(t)-K_d(0)) \frac{d^d}{dt^d} e^{(\sigma-s)t} dt  \\
  &= n_{\sigma} (-1)^d \frac{(\sigma -s)^d}{d!}\int_0^{+\infty} t^d e^{(\sigma-s)t} \, dt + \\
&  \quad +\sum_{\rho} \frac{n_\rho}{(\rho-\sigma)^d} (-1)^d (\sigma-s)^d 
  \left( \int_0^{+\infty} e^{(\rho-\sigma)t}e^{(\sigma-s)t} dt  
   - \int_0^{+\infty} e^{(\sigma-s)t} dt  \right) \\
  &= \frac{n_{\sigma}}{s-\sigma}-\sum_{\rho} n_\rho \frac{ (s- \sigma)^d}{(\rho-\sigma)^d} \left( \frac{1}{\rho-s} - \frac{1}{\sigma-s} \right) \\
   &= \frac{n_{\sigma}}{s-\sigma}-\sum_{\rho} n_\rho \frac{ (s- \sigma)^{d-1}}{(\rho-\sigma)^{d-1}}  \frac{1}{\rho-s} 
  =  G(s) 
 = \frac{f'(s)}{f(s)} + P_{f,\sigma}(s)\, .
  \end{align*}

By uniqueness of the Laplace transform for Laplace transformable distributions
(see \cite[Theorem 8.3-1]{Z}), we have
  $$
  W_\sigma(f)= \cL^{-1}(f'/f) + \cL^{-1}(P_{f,\sigma}),
  $$
where $\cL^{-1}(P_{f,\sigma})$ is the inverse Laplace 
transform of the polynomial $P_{f,\sigma}$. This is a
distribution supported at $\{0\}$. If $P_{f,\sigma}(s)=c_0  +c_1 s +\ldots +c_{g-1} s^{g-1}$, then
 $$
\cL^{-1} (P_{f,\sigma})=c_0 \delta_0 +c_1 \delta'_0 +\ldots +c_{g-1} \delta_0^{(g-1)}  \ ,
 $$
where $c_j=c_j(\sigma)$.
\end{proof}

Note in particular that 
  $$
  W(f)=W_\sigma(f)|_{\RR_+^*}= \cL^{-1}(f'/f)|_{\RR_+^*} \, .
  $$

The inverse Laplace transform $\cL^{-1}(F)$ is a well defined distribution of finite order 
when $F$ has polynomial growth on a half plane \cite[Theorem 8.4-1]{Z}. For $f$ a
meromorphic function of finite order whose divisor is contained on a left half plane,
$F=f'/f$ has polynomial growth on a half plane (see \cite{MPM-genus}), hence
$\cL^{-1}(F)$ is well-defined (although this is also clear from
theorem \ref{thm:Poisson-Newton}).

Let us recall how to compute $\cL^{-1}(F)$ from \cite[p.\ 236]{Z}.
Take $m_0$ which is two units more than the order of growth of $F$, that is
$F(s)|s|^{-m_0} \leq C/|s|^2$ for $\Re s\geq  \sigma_2 >0$. Define
 $$
 L(t)=\int_{-\infty}^{+\infty} \frac{F(c+iu)}{(c+iu)^{m_0}} e^{(c+iu)t} \, \frac{du}{2\pi} \, .
  $$
This is a continuous function, which vanishes on $\RR_-$. It is independent of the choice of $c$
(subject to $c>\max\{\sigma_1,\sigma_2\}$ and for $F$ holomorphic).
Then
$$
\cL^{-1}(F) (t) = \frac{D^{m_0}}{Dt^{m_0}} L(t),
$$
which is a distribution of order at most $m_0-2$.

More explicitly, for an appropriate test function $\varphi$, letting $\psi (t)=\varphi (t) e^{ct}$, we have
\begin{equation}\label{eqn:Laplace}
\begin{aligned}
\langle \cL^{-1}(F) , \varphi \rangle & = \langle L(t), (-1)^{m_0} \varphi^{(m_0)}(t)\rangle \\
 &= \int_\RR 
\int_{-\infty}^{+\infty}\frac{F(c+iu)}{(c+iu)^{m_0}} (-1)^{m_0} \varphi^{(m_0)}(t) e^{(c+iu)t} \, \frac{du}{2\pi} dt \\
&= \int_{-\infty}^{+\infty} (-1)^{m_0} \frac{F(c+iu)}{(c+iu)^{m_0}} 
 \left( \int_\RR  \varphi^{(m_0)}(t) e^{(c+iu)t}  dt \right) \, \frac{du}{2\pi}  \\
 &= \int_{-\infty}^{+\infty} \frac{F(c+iu)}{(c+iu)^{m_0}} 
 \left( \int_\RR  (c+iu)^{m_0} \varphi(t) e^{(c+iu)t}  dt \right) \, \frac{du}{2\pi}  \\
 &= \int_{-\infty}^{+\infty} 
  \int_\RR F(c+iu) \varphi(t) e^{(c+iu)t}  dt \, \frac{du}{2\pi}  \\
 &=\int_{-\infty}^{+\infty} F(c+iu)
 \hat \psi(-u) \, \frac{du}{2\pi}  \, ,
\end{aligned}
\end{equation}
doing $m_0$ integrations by parts in the fourth line.

\medskip

We can give a symmetric version of the Poisson-Newton formula for
 $$
 \widehat{W}(f)=\sum n_\rho e^{\rho|t|}.
$$

 \begin{theorem}\label{thm:symmetric-for-general}
 We have, as distributions on $\RR$,
$$
\widehat{W}_\sigma(f)= 2 \sum_{l=0}^{\frac{g-1}{2}} c_{2l}(\sigma) \, \delta_0^{(2l)} + 
 \big( \cL^{-1}(f'/f)(t) + \cL^{-1}(f'/f)(-t) \big).
$$
 \end{theorem}

\begin{proof}
 It follows from the definition $\widehat{W}_\sigma(f)(t)={W}_\sigma(f)(t) +
{W}_\sigma(f)(-t)$ and theorem \ref{thm:Poisson-Newton}.
Note that doing the change of variables $t\mapsto -t$ on
$\sum_{l=0}^{g-1} c_{l} \,  \delta_0^{(l)}$, we get
$\sum_{l=0}^{g-1} (-1)^l c_{l} \, \delta_0^{(l)}$.
\end{proof}

Furthermore, we have a parameter version of our main theorem by doing the
change of variables
$s\mapsto \alpha s +\beta$, with $\alpha >0$ and $\beta \in \CC$. 
Take $\sigma'=\frac{\sigma-\beta}{\alpha}$. 
We denote, as a slight abuse of notation, 
$$
 e^{-\frac{\beta}{\alpha}|t|} \widehat{W}_\sigma(f)(t)=
 e^{-\frac{\beta}{\alpha} t} {W}_\sigma(f)(t) + e^{\frac{\beta}{\alpha} t} W_\sigma(f)(-t)\, .
$$
(multiplication of a distribution by a non-smooth function is not defined in 
general, so we have to give an explicit meaning to the left hand side). Note that formally,
$$
e^{-\frac{\beta}{\alpha} |t|} \widehat{W}_\sigma(f)(t/\alpha)
=\sum_\rho n_\rho e^{(\rho-\beta) |t|/\alpha}\, .
$$

 \begin{corollary}\label{cor:thm:symmetric-parameters-general}
 We have the equality of distributions on $\RR$
$$
 e^{-\frac{\beta}{\alpha} |t|} \widehat{W}_\sigma(f)(t/\alpha)= 
  2 \sum_{l=0}^{\frac{g-1}{2}} c_{2l}' \, \delta_0^{(2l)} + 
  \big( e^{-\frac{\beta}{\alpha} t} \cL^{-1}(f'/f)(t/\alpha) + 
 e^{\frac{\beta}{\alpha} t} \cL^{-1}(f'/f)(-t/\alpha) \big),
$$
 for some $c_j'$ explicitly determined in the proof below.
 \end{corollary}

\begin{proof}
Consider the meromorphic function $g(s)=f(\alpha s+\beta)$, which
has all its zeros in a half-plane. The zeroes of $g$ are $\left( 
(\rho-\beta)/\alpha\right)$. 
By the definition (\ref{eqn:N-C}), we have
 $$
 W_{\sigma'}(g)(t)= e^{-\frac{\beta}{\alpha} t}  W_\sigma(f)(t/\alpha).
 $$
The discrepancy polynomials satisfy
$Q_{g,\sigma'}(s)= \textit{constant} + Q_{f,\sigma}(\alpha s+\beta)$. Therefore
 $$
 P_{g,\sigma'}(s)= \alpha P_{f,\sigma}(\alpha s+\beta)
 $$
Write $c_0'+c_1's+\ldots +c_{g-1}' s^{g-1}= 
\alpha(c_0+c_1 (\alpha s+\beta) +\ldots +c_{g-1} (\alpha s+\beta) ^{g-1})$. 
The penultimate line of equation (\ref{eqn:Laplace}) implies that 
 $$
 \cL^{-1}(g'/g)(t) = e^{-\frac{\beta}{\alpha} t} \cL^{-1}(f'/f)(t/\alpha).
 $$
Then theorem \ref{thm:symmetric-for-general} applied to $g$ implies the result.
\end{proof}

We end up with an application for a real analytic function. Note that
for a non-holomorphic function $h$, with polynomial decay in the right half-plane,
the Laplace transform depends on the line of integration. We denote
$ \cL^{-1}_\beta(h)$ for the integration along $\Re s=\beta$, with $\beta$ large enough.

 \begin{corollary}\label{cor:symmetric-parameters-general}
 For a real analytic function $f$ and $\beta\in \RR$ to the right of all zeroes of $f$, we have
as distributions on $\RR$,
$$
 e^{-\frac{\beta}{\alpha} |t|} \widehat{W}_\sigma(f)(t)= 
  2 \sum_{l=0}^{\frac{g-1}{2}} c_{2l}'  \, \delta_0^{(2l)} + 
  e^{-\frac{\beta}{\alpha} t} \cL^{-1}_\beta  \left( 2 \Re(f'/f)\right) (t/\alpha).
$$
 \end{corollary}
 
\begin{proof}
We only have to prove that, for a real analytic function $F$, 
and $\gamma$ to the right of the zeroes, we have
$$
  e^{-\gamma t} \cL^{-1}(F)(t) + 
 e^{\gamma t} \cL^{-1}(F)(-t) = 
  e^{-\gamma t} \cL^{-1}_\gamma  \left( 2 \Re F\right) (t).
  $$

By (\ref{eqn:Laplace}), we have
 $$
 \la e^{-c t} \cL^{-1}(F)(t), \varphi(t) \ra= \int_{-\infty}^{+\infty} 
F(c+i u) \hat\varphi (-u) \frac{du}{2\pi}
 $$
Analogously,
\begin{align*}  
  \la e^{c t} \cL^{-1}(\overline F)(-t),\varphi(t)\ra &= 
  \la e^{-c t} \cL^{-1}(\overline F)(t),\varphi(-t)\ra 
  =  \int_{-\infty}^{+\infty} \overline{F}(c+i u) \overline{\hat\varphi(-u)}  \frac{du}{2\pi} \\
  &=  \int_{-\infty}^{+\infty} \overline{F(c-i u)} \hat\varphi (u) \frac{du}{2\pi}
  =  \int_{-\infty}^{+\infty} \overline{F(c+i v)} \hat\varphi (-v) \frac{dv}{2\pi} \, .
 \end{align*}
 Adding both, we get
  $$
 \int_{-\infty}^{+\infty} 2 \left( \Re F(c+iv ) \right) \hat\varphi (-v) \frac{dv}{2\pi}  =
 \la e^{-ct} \cL_c^{-1}\left( 2 \Re F \right) ,\varphi(t)\ra,
  $$
 as required.
  \end{proof}

For later use, we also need to 
determine the relation between $Q_{f,\sigma}$ and $Q_{f,0}=Q_f$. In particular,  
the relation between the coefficients $c_0(\sigma)$ and $c_0(0)=c_0$. 
Let $f$ be of finite order and consider the Hadamard factorization of $f$
$$
f(s)=s^{n_0} e^{Q_f(s)} \prod_{\rho \not=0 } E_m (s/\rho )^{n_\rho} \ ,
$$
and the corresponding Hadamard factorization centered at $\sigma \in\CC$,
$$
f(s)=(s-\sigma)^{n_\sigma} e^{Q_{f,\sigma}(s)} \prod_{\rho \not=\sigma } E_m \left(\frac{s-\sigma}{\rho-\sigma} 
\right)^{n_\rho} \, .
$$

\begin{lemma} For $d=2$ we have
 \begin{align} \label{eqn:c4}
  c_0(\sigma)=c_0+\frac{n_0}{\sigma} + \frac{n_\sigma}{\sigma} + \sum_{\rho\neq 0,\sigma} n_\rho 
\frac{-\sigma}{\rho(\rho-\sigma)}
 \end{align}
\end{lemma}

\begin{proof}
We need to understand the difference between these two factorizations. We take logarithmic derivatives to get
 \begin{align*}
  & \frac{n_\sigma}{s-\sigma} + Q'_{f,\sigma} +  \sum_{\rho\neq 0,\sigma} n_\rho 
\frac{(s-\sigma)^m}{(\rho-\sigma)^m} \frac{1}{s-\rho}
 + n_0 \frac{(s-\sigma)^m}{(-\sigma)^m} \frac{1}{s} = \\
& = \frac{n_0}{s} + Q'_{f} +  \sum_{\rho\neq 0,\sigma} n_\rho 
\frac{s^m}{\rho^m} \frac{1}{s-\rho}
 + n_\sigma \frac{s^m}{\sigma^m} \frac{1}{s-\sigma}
 \end{align*}
Therefore
 \begin{align*}
  Q'_{f,\sigma}  -   Q'_{f} = & n_0\frac{(-\sigma)^m -(s-\sigma)^m}{(-\sigma)^m s} + n_\sigma
\frac{s^m -\sigma^m}{\sigma^m (s-\sigma)}  \\
&+ \sum_{\rho\neq 0,\sigma} n_\rho 
\frac{s^m(\rho-\sigma)^m- (s-\sigma)^m \rho^m}{\rho^m(\rho-\sigma)^m} \frac{1}{s-\rho}
 \end{align*}
For $m=1$ this reduces to 
 $$
Q'_{f,\sigma}  -   Q'_{f} = \frac{n_0}{\sigma} + \frac{n_\sigma}{\sigma} + \sum_{\rho\neq 0,\sigma} n_\rho 
\frac{-\sigma}{\rho(\rho-\sigma)}\ .
$$
\end{proof}

\section{Dirichlet series}\label{sec:dirichlet}

We consider a non-constant Dirichlet series
 \begin{equation}\label{eqn:1}
 f(s)=1+\sum_{n\geq 1} a_n \ e^{-\lambda_n s} \ ,
 \end{equation}
with $a_n \in \CC$ and
$$
0< \lambda_1 < \lambda_2 < \ldots
$$
with  $\lambda_n \to +\infty$ or $(\lambda_n)$ is a finite set (equivalently, take
the sequence $(a_n)$ with all but finitely many elements being zero). Suppose
that we have a half
plane of absolute convergence (see \cite{HR} for background on Dirichlet series), i.e., for some $\bar \sigma \in \RR$ we have
$$
\sum_{n\geq 1} |a_n | \ e^{-\lambda_n \bar \sigma} <+\infty \, .
$$
It is classical \cite[p.\ 8]{HR} that
$$
\bar \sigma=\limsup \frac{\log (|a_1|+|a_2|+\ldots +|a_n|)}{\lambda_n} \ .
$$

The Dirichlet series (\ref{eqn:1}) is therefore absolutely and uniformly
convergent on right half-planes $\Re s \geq \sigma$, for any $\sigma >\bar \sigma$.

We assume that $f$ has a meromorphic extension of finite order
to all the complex plane $s\in \CC$. We denote by $(\rho)$ the set of zeros and poles
of $f$, and the integer $n_\rho$ is the multiplicity of $\rho$. We have, uniformly on $\Re s$,
$$
\lim_{\Re s \to +\infty} f(s) =1 \ ,
$$
thus 
$$
\sigma_1 =\sup_\rho \Re \rho <+\infty \ ,
$$
so $f(s)$ has neither zeros nor poles on the half plane $\Re s >  \sigma_1$. 
Sometimes in the applications $\sigma_1$ is a pole of $f$ because when the coefficients $(a_n)$ are real and positive 
then $f$ contains a singularity at $\bar \sigma$ by a classical theorem of Landau (see \cite[Theorem 10]{HR}). 
The singularity is necessarily a pole by our assumptions, and in general $\sigma_1 =\bar \sigma$.

On the half plane $\Re s >  \sigma_1$, $\log f(s)$ is well
defined taking the principal branch of the logarithm. Then we can
define the coefficients $(b_{\bk})$ by 
\begin{equation} \label{eqn:bn}
-\log f(s)=-\log \left ( 1+ \sum_{n\geq 1} a_n \ e^{-\lambda_n s}\right )
=\sum_{\bk \in \Lambda} b_{\bk} \, e^{-\langle \boldsymbol{\lambda} , \bk \rangle s}
 \ ,
 \end{equation}
where $\Lambda=\{ \bk=(k_n)_{n\geq 1} \, | \, k_n \in \NN, ||\bk||=\sum | k_n |<\infty, ||\bk|| \geq 1\}$,
and
$\langle \boldsymbol{\lambda} , \bk \rangle = \lambda_1k_1+\ldots + \lambda_{l}k_{l}$, where
$k_n=0$ for $n>l$.
Note that the coefficients $(b_{\bk})$ are polynomials on the $(a_n)$. More precisely, we have
\begin{equation} \label{eqn:bs}
 b_\bk= \frac{(-1)^{||\bk||}}{||\bk||} \, \frac{||\bk|| !}{\prod_j k_j!}\, \prod_j a_j^{k_j}\, .
\end{equation}

Note that if the $\lambda_n$ are $\QQ$-dependent then there are repetitions in
the exponents of (\ref{eqn:bn}).

The main result is the following Poisson-Newton formula associated to the Dirichlet series $f$.

\begin{theorem}  \label{thm:main}
 As Laplace transormable distributions in $\RR$ we have
 $$
  W_\sigma (f)= \sum_{k=0}^{g-1} c_k(\sigma) 
 \delta_0^{(k)} + \sum_{\bk \in \Lambda} \langle \lambda , \bk
 \rangle \, b_{\bk} \ \delta_{\langle \boldsymbol{\lambda} ,\bk\rangle } \, .
 $$
\end{theorem}

\begin{proof}
By theorem \ref{thm:Poisson-Newton}, we only have to prove that the (right) Laplace transform
of the distribution 
  $$
   V=\sum_{\bk} \langle \boldsymbol{\lambda} , \bk\rangle \, b_{\bk} \ \delta_{\langle \boldsymbol{\lambda} , \bk\rangle}
   $$
equals $f'/f$. We compute
 $$
   \la V, e^{-ts}\ra_{\RR_+}
   = \left \la \sum_{\bk} \langle \boldsymbol{\lambda} , \bk\rangle \, b_{\bk} \ \delta_{\langle \boldsymbol{\lambda} , \bk\rangle},
e^{-ts} \right \ra_{\RR_+} 
   =\sum_{\bk} \langle \boldsymbol{\lambda} , \bk\rangle \, b_{\bk} e^{-\langle \boldsymbol{\lambda} , \bk \rangle s} 
   = (\log f(s))' =\frac{f'(s)}{f(s)}\, ,
 $$
as required.
\end{proof}

 Recalling that $W(f)=W_\sigma (f)|_{\RR_+^*}$, we have
the Poisson-Newton formula on $\RR_+^*$,
 \begin{equation} \label{eqn:P-N-R+}
  W(f)= \sum_{\bk \in \Lambda} \langle \boldsymbol{\lambda} , \bk
\rangle \, b_{\bk} \ \delta_{\langle \boldsymbol\lambda ,\bk\rangle } \, ,
 \end{equation}
as distributions on $\RR^*_+$.

Consider a Dirichlet series $f(s)=1+\sum a_n e^{\lambda_n s}$ and let 
 $$
 \bar f (s)=\overline{f(\bar s)}=1+\sum \bar a_n e^{\lambda_n s}
 $$
be its conjugate. Then $\bar f$ is a Dirichlet series whose 
zeros are the $\{\bar \rho\}$ and $n_{\bar \rho} =n_\rho$. 
Also $b_\bk (\bar f)=\overline{b_\bk (f)}$.
The Poisson-Newton formula for $\bar f$ is 
$$
W_{\bar \sigma}(\bar f)(t)= \sum_{l=0}^{g-1} 
\bar c_{l} \,  \delta_0^{(l)} +
\sum_{\bk \in \Lambda} \langle \boldsymbol{\lambda} , {\bk}  \rangle 
\overline{b_\bk} \, \delta_{\langle \boldsymbol{\lambda} ,{\bk}\rangle }\, .
$$

\begin{corollary}\label{real-analytic:symmetric}
 For a real analytic Dirichlet series $f$, that is $\bar f (s)=f(s)$, we have that
 for $\sigma\in \RR$, the numbers $c_l$ and $b_\bk$ are real.
  The converse also holds.
\end{corollary}

\begin{proof}
For a real analytic Dirichlet series, $a_n$ are real. Hence $b_\bk$ are real numbers.
The discrepancy polynomial has also real coefficients, so the $c_l$ are real.
The last point is due to the fact that the association $f\mapsto W(f)$ is one-to-one, 
as its inverse is the Laplace transform.
\end{proof}

We also have a symmetric Poisson-Newton formula for
$$
\widehat{W}_\sigma(f)(t)= W_\sigma(f)(t) +W_\sigma(f)(-t)=
\sum_\rho n_\rho e^{\rho |t|}  \ .
$$

\begin{theorem}\label{thm:symmetric}
 For a Dirichlet series $f$, we have as distributions on $\RR$,
$$
\widehat{W}_\sigma(f)(t)= 2 \sum_{l=0}^{\frac{g-1}{2}} c_{2l}(\sigma) \, \delta_0^{(2l)} + 
\sum_{\bk \in \Lambda \cup (-\Lambda )} \langle \boldsymbol{\lambda} , |\bk |
\rangle \, b_{|\bk |} \ \delta_{\langle \boldsymbol{\lambda} ,\bk\rangle } \, ,
 $$
where we denote $|\bk|=-\bk$, for $\bk \in -\Lambda$. \hfill $\Box$
\end{theorem}

\medskip

For completeness, we also give parameter versions of the Poisson-Newton
formulas for Dirichlet series. Observe that the space of Dirichlet series is invariant by
the change of variables
$s\mapsto \alpha s +\beta$, with $\alpha >0$ and $\beta \in \CC$.
We have the following Poisson-Newton formula for 
$$
e^{-\frac{\beta}{\alpha} |t|} \widehat{W}_\sigma(f)(t/\alpha)
=\sum_\rho n_\rho e^{(\rho-\beta) |t|/\alpha}\, .
$$

\begin{corollary}\label{cor:thm:symmetric_parameters} 
Let $\alpha >0$ and $\beta \in \RR$.
We have as distributions on $\RR$,
$$
e^{-\frac{\beta}{\alpha} |t|} \widehat{W}_{\sigma}(f)(t/\alpha)= 
 2 \sum_{l=0}^{\frac{g-1}{2}}  c_{2l}'
 \, \delta_0^{(2l)} + \sum_{\bk \in \Lambda \cup (-\Lambda )} \alpha \, \langle \boldsymbol{\lambda} , |\bk |
\rangle e^{-\langle \boldsymbol{\lambda} , |\bk |\rangle \beta}   \, b_{|\bk |} \ \delta_{\alpha \langle \boldsymbol{\lambda} ,\bk\rangle } \, ,
$$
where $c_j'$ are given by $c_0'+c_1's+\ldots +c_{g-1}' s^{g-1}= 
\alpha(c_0+c_1 (\alpha s+\beta) +\ldots +c_{g-1} (\alpha s+\beta) ^{g-1})$.
\end{corollary}

\begin{proof}
This results by applying corollary \ref{cor:thm:symmetric-parameters-general} to the 
Dirichlet series $f(s)$.
\end{proof}

In particular, for $\alpha =1$ and $g=1$ that we use in the applications, we get
$$
\sum_\rho n_\rho e^{(\rho-\beta) |t|}= 2  c_0(\sigma) \, 
\delta_0 + \sum_{\bk \in \Lambda \cup (-\Lambda )}  \, \langle \boldsymbol{\lambda} , |\bk |
\rangle e^{-\langle \boldsymbol{\lambda} , |\bk |\rangle \beta}   
\, b_{|\bk |} \ \delta_{ \langle \boldsymbol{\lambda} ,\bk\rangle } \, .
$$
Here we determine $c_0'$ explicitly: we have $g(s)=f(s+\beta)$ and
$\sigma'= \sigma-\beta$. Then $c_0'=c_0'(g,\sigma')=c_0(f,\sigma)$.

\section{Basic applications}

\subsection{Classical Poisson formula} \label{subsec:classical}

The Poisson-Newton formula that we have proved in section \ref{sec:dirichlet} 
is a generalization of the 
well-known classical Poisson formula
\begin{equation}\label{eqn:classical}
\sum_{k\in \ZZ} e^{i\frac{2\pi}{\lambda} k t}  =\lambda \sum_{k\in \ZZ}  \delta_{\lambda k} \, .
\end{equation}
where $\lambda>0$. Actually this Poisson formula is associated to a Dirichlet series
of only one frequency.

We derive the classical Poisson formula from the symmetric Poisson-Newton formula.
It is also interesting to clarify the structure of the
Newton-Cramer distribution at $0$. It helps to understand why the 
Dirac $\delta_0$ appearing in the right side of the
classical Poisson formula is of a different nature than the other $\delta_{\lambda k}$ for $k\not= 0$, 
something that was intuitively suspected from the analogy with trace formulas
(see a comment on this in \cite[p.\ 2]{CV}).

In order to use the symmetric Poisson-Newton formula we compute the discrepancy polynomial $P_f$ for
$f(s)=1-e^{-\lambda s}$. We have that $\sigma=0$ is a zero. From the classical Hadamard factorization
$$
\sinh (\pi s)=\pi s \prod_{k\in \ZZ^*} \left (1-\frac{s}{ik}\right ) e^{\frac{s}{ik}} \ ,
$$
we get the Hadamard factorization for $f$,
$$
f(s)=2 e^{-\lambda s /2} \sinh(\lambda s/2)=s \lambda e^{-\lambda s/2} \prod_{k\in \ZZ^*} \left (1-\frac{s}{\rho_k}\right ) e^{\frac{s}{\rho_k}} \ .
$$
Note that this is equivalent to
 $$
G(s) =\frac{1}{s} - \sum_{k\in \ZZ^*} \left (\frac{1}{\rho_k-s}-\frac{1}{\rho_k} \right )  
=\frac{\lambda/2}{\tanh\left (\lambda s/2\right )} \ .
 $$

Thus $Q_f(s)=(\log \lambda +2\pi i n)-\frac{\lambda}{2} s$, with $n\in \ZZ$, and
$$
P_f(s)=-Q_f'(s)=c_0=\frac{\lambda}{2} \ .
$$

Therefore we can apply the symmetric Poisson-Newton formula (theorem \ref{thm:symmetric}) and we get
 $$
\sum_{k\in \ZZ} e^{i\frac{2\pi}{\lambda} k|t|} = 2 c_0 \delta_0 +
\lambda \sum_{k\in \ZZ^*} \delta_{\lambda k} 
= \lambda \delta_0 + \lambda \sum_{k\in \ZZ^*} \delta_{\lambda k} 
=\lambda \sum_{k\in \ZZ} \delta_{\lambda k} \ .
 $$
We finally observe that
$$
\sum_{k\in \ZZ} e^{i\frac{2\pi}{\lambda} k|t|} = 1+2\sum_{k=1}^{+\infty} \cos \left (\frac{2\pi}{\lambda} k |t|
\right ) = 1+2\sum_{k=1}^{+\infty} \cos \left (\frac{2\pi}{\lambda} k t
\right ) =\sum_{k\in \ZZ} e^{i\frac{2\pi}{\lambda}k t }\ ,
$$
because we can reorder freely a converging (in the distribution sense) infinite series of distributions.

\subsection{Newton formulas}

We show now how the Poisson-Newton formula is a generalization to Dirichlet series of Newton
formulas which express Newton sums of roots of a polynomial equation in terms of its coefficients
(or  elementary symmetric functions).

Let $P(z)=z^n+a_1 z^{n-1}+\ldots +a_n $ be a polynomial of degree $n\geq 1$ with zeros $\alpha_1, \ldots, \alpha_n$
repeated according to their multiplicity. 
For each integer $m\geq 1$, the Newton sums of the roots are the symmetric functions
$$
S_m=\sum_{j=1}^n \alpha_j^m \ .
$$
From the fundamental theorem on symmetric functions, these Newton sums can be expressed polynomially with
integer coefficients in terms of
elementary symmetric functions, i.e., in terms of the coefficients of $P$. These are the Newton formulas. For instance, if for $k\geq 1$
$$
\Sigma_k =\sum_{1\leq i_1<\ldots < i_k \leq n} \alpha_{i_1}\ldots \alpha_{i_k} = (-1)^k a_k \ ,
$$
then we have
\begin{align*}
S_1&= \Sigma_1 \\
S_2&= \Sigma_1^2 - 2 \Sigma_2 \\
S_3&= \Sigma_1^3 -3\Sigma_2 \Sigma_1 +3\Sigma_3 \\
S_4&= \Sigma_1^4 -4\Sigma_2\Sigma_1^2 +4 \Sigma_3 \Sigma_1 +2\Sigma_2^2 -4\Sigma_4 \\
&\vdots
\end{align*}

We recover
them applying the Poisson-Newton formula to the finite Dirichlet series
$$
f(s)=e^{-\lambda ns }P(e^{\lambda s}) =1+a_1 e^{-\lambda s} +\ldots + a_n e^{-\lambda n s} \ .
$$
The zeros of $f$ are the $(\rho_{j,k})$ with $j=1,\ldots , n$, $k\in \ZZ$, and
$$
e^{\rho_{j, k}}= \alpha_j^{1/\lambda} e^{\frac{2\pi i}{\lambda}k} \ .
$$
Thus, using the classical Poisson formula (with $\sigma=0$), its Newton-Cramer distribution can be computed in $\RR$ as
\begin{align*}
 \sum_\rho e^{\rho |t|} &= \sum_{j=1}^n \alpha_j^{(1/\lambda) t} \sum_{k\in \ZZ} e^{\frac{2\pi i}{\lambda}kt} 
=\sum_{j=1}^n \alpha_j^{(1/\lambda) t} \lambda \sum_{m\in \ZZ} \delta_{m\lambda}\\
&=\lambda \sum_{m\in \ZZ} \left ( \sum_{j=1}^n \alpha_j^{m} \right )\, \delta_{m\lambda} 
= \lambda \sum_{m\in \ZZ} S_m \, \delta_{m\lambda}\ .
\end{align*}

Now, using the Poisson-Newton formula in $\RR_+^*$
$$
 \sum_\rho e^{\rho t}= \sum_{\bk \in \Lambda} \langle \boldsymbol{\lambda} , \bk
\rangle \, b_{\bk} \ \delta_{\langle \boldsymbol{\lambda} ,\bk\rangle } \, ,
 $$
taking into account the repetitions in the right side, and that
$\boldsymbol{\lambda}=(\lambda_1,\ldots, \lambda_n)= (\lambda ,2\lambda ,\ldots, n\lambda )$,
we have, using the formula (\ref{eqn:bs}) for the $b_\bk$,
$$
S_m=m \sum_{ k_1+2k_2+\ldots +nk_n=m} b_\bk = m \sum_{ k_1+2k_2+\ldots +nk_n=m} \frac{(||k||-1)!}{\prod_j k_j} \prod_j \Sigma_j^{k_j} \, ,
$$
which gives the explicit Newton relations. Moreover, Newton relations are equivalent to the Poisson-Newton formula in $\RR_+^*$ in this case. 

For example, for $m=4$,
$$
S_4=4\, (b_{(4,0,0,0)} +b_{(2,1,0,0)} + b_{(1,0,1,0)} + b_{(0,2,0,0)} + b_{(0,0,0,1)} ) \ ,
$$
and from $b_{(4,0,0,0)} = \frac14 \Sigma_1^4$, $b_{(2,1,0,0)} =  - \Sigma_1^2 \Sigma_2$,
 $b_{(1,0,1,0)} = \Sigma_1 \Sigma_3$,
 $b_{(0,2,0,0)}= \frac12 \Sigma_2^2$ and $b_{(0,0,0,1)} = - \Sigma_4$,
we get 
$$
S_4= \Sigma_1^4 -4\Sigma_2\Sigma_1^2 +4 \Sigma_3 \Sigma_1 +2\Sigma_2^2 -4\Sigma_4 \ .
$$

\section{Functional equations} \label{sec:functional}

When a Dirichlet series satisfies a functional equation, we can deduce a constraint on the structure
at zero of $W_\sigma(f)$ for $\sigma$ the center of symmetry (theorem \ref{thm:functional_eq_0structure}). 
We start by giving a precise definition of the property of ``having a functional equation'', as we know of no reference
in the classical literature. We start with a simple remark.

\begin{lemma} \label{lemma:cone}
For $\theta_1 < \theta_2 $, $\theta_2-\theta_1 < \pi$, denote by $C(\theta_1, \theta_2)$ the cone
of values of $s\in \CC$ with
$\theta_1 < \Arg \ s < \theta_2$. If $\{\rho\}$ is contained in a cone $\alpha+ C(\theta_1, \theta_2)$, 
$\alpha \in \RR$, then
$$
W(f)(t)=\sum_\rho n_\rho e^{\rho t} \ ,
$$
is a holomorphic function on the variable $t$ in $C(\pi/2-\theta_1 , 3\pi/2 -\theta_2)$.
\end{lemma}

\begin{proof}
For $t\in C(\pi/2-\theta_1 , 3\pi/2 -\theta_2)$ we have $\Re ( (\rho -\alpha)\, t) <0$, whence
$$
\left |e^{\rho t} \right | <e^{\alpha t} \ ,
$$
and the series $K_\ell(t)$ defined in (\ref{eqn:Kl}) is holomorphic in that region, so the result follows.
\end{proof}

From this we obtain the following straightforward corollary:

\begin{corollary} \label{cor:divisor}
 The divisor of any Dirichlet series cannot be contained in a cone $\alpha+ C(\theta_1, \theta_2)$ for
$\pi/2 <\theta_1 < \theta_2 <3\pi/2 $.
\end{corollary}

\begin{proof}
From the Poisson formula (\ref{eqn:P-N-R+}), 
we get that the distribution $W(f)$ is an atomic distribution on $\RR_+^*$.
Thus the sum of exponentials associated to the zeros cannot be a convergent series for $t\in \RR_+^*$. 
But lemma \ref{lemma:cone} gives the analytic convergence of the sum if
the divisor is contained in the cone $\alpha+ C(\theta_1, \theta_2)$ .
\end{proof}

We say that the divisor $D_1$ is contained in the divisor $D_2$, and denote this by
$D_1 \subset D_2$ if any zero, resp.\ pole,  of $D_1$ is a zero,
resp.\ pole, of $D_2$, and $|n_\rho (D_1)| \leq |n_\rho (D_2)|$ for all $\rho \in \CC$.
Also if $D_1=\sum n_\rho \,\rho$ and $D_2=\sum m_\tau \, \tau$ are two divisors, then
the sum and difference are defined by 
$D_1+D_2=\sum n_\rho \,\rho + \sum m_\tau \, \tau$
and $D_1-D_2=\sum n_\rho \,\rho + \sum (-m_\tau) \, \tau$.

\begin{definition} \label{def:functional}
 The meromorphic function $f$ has a functional equation if there exists $\sigma^*\in \RR$ and a
divisor $D\subset \Div (f)$ contained in a left cone
$\alpha+ C(\theta_1, \theta_2)$, with $\pi/2 <\theta_1 < \theta_2 <3\pi/2 $, 
such that $\Div (f) -D$ is infinite and symmetric with respect to the vertical line $\Re s =\sigma^*$.
\end{definition}

\begin{proposition}
 If $f$ has a functional equation and 
the divisor of $f$ is contained in a left half plane then $\sigma^* \in \RR$ is unique.
\end{proposition}

\begin{proof}
Otherwise, if they were two distinct values $\sigma^*$, then $\Div (f)$ would have an infinite 
subdivisor invariant by a real 
translation and this contradicts the hypothesis that the divisor of $f$ is contained in a left half plane.
\end{proof}

\begin{proposition} \label{prop:6.4}
 If $f$ has a functional equation and the divisor of $f$ is contained in a 
left half plane then $\Div (f) -D$ is contained 
in a vertical strip.  The minimal strip $\{ \sigma_- <\Re s < \sigma_+ \}$,
$\sigma_+ \leq \sigma_1$, with this property is the critical strip and
$\sigma^*=\frac{\sigma_-+\sigma_+}{2}$ is its center.
\end{proposition}

\begin{proof}
Since $\Div (f)$ has no zeros nor poles for $\Re s >\sigma_1$, the divisor of $\Div (f) -D$ is contained
in a vertical strip due to the symmetry. The minimal vertical strip has to be compatible with the functional equation, 
hence $\sigma^*=\frac{\sigma_-+\sigma_+}{2}$.
\end{proof}

\begin{proposition}
 If $f$ has a functional equation and the divisor of $f$ is contained in a left half plane then there is a unique 
minimal divisor $D$ (i.e., with $ |n_\rho (D) |$ minimal for all $\rho \in \CC$), and a unique decomposition 
$D=D_0+D_1$, $D_0$ and $D_1$ with disjoint supports, with $D_0$ contained in a left cone
$\sigma^*+C(\theta_1, \theta_2)$, with $\pi/2 <\theta_1 < \theta_2 <3\pi/2 $, and $D_1$ a finite divisor 
contained in the half plane $\Re s >\sigma^*$, such that $\Div ( f) -D$ is infinite and symmetric with 
respect to the vertical line $\Re s =\sigma^*$.
\end{proposition}

\begin{proof}
We start with $D$ minimal as in the definition, and we define $D_0$ to be the part of $D$ to the left of  
$\Re s =\sigma^*$ and $D_1$ the remaining part. It is easy to see that $D_0$ is contained in a left cone with 
vertex at $\sigma^*$.
\end{proof}

 \begin{proposition}
 If $f$ has a functional equation and the divisor of $f$ is contained in a left half plane then there exists 
a meromorphic function $\chi$ with $\Div (\chi) =D \subset \Div (f)$, such that the function $g(s)=\chi(s) f(s)$ satisfies
the functional equation 
 $$
g(2\sigma^* - s)=g(s) \ .
 $$
Moreover, we can write $\chi =\chi_0 \cdot R$ with $\Div (\chi_0) =D_0-\tau^*D_1$ and $\Div (R) =D_1 + \tau^*D_1$,
where $\tau$ is the reflexion along $\Re s=\sigma^*$,  
and $R$ is a unique rational function up to a non-zero multiplicative constant.

The meromorphic function $\chi$ (or $\chi_0$) is uniquely determined
up to a factor $\exp h(s-\sigma^*)$ where $h$ is an even entire function. 
If $f$ has convergence exponent $d <+\infty$,  
then  we can take $\chi$  of  convergence exponent $d$, and then $\chi$ 
is uniquely determined up to a factor $\exp P(s-\sigma^*)$
where $P$ is an even polynomial of degree
less than $d$. In particular, when $f$ is of order $1$ then $\chi$ 
and $\chi_0$ are uniquely determined up to a
non zero multiplicative constant.
\end{proposition}

\begin{proof}
 {}From proposition \ref{prop:6.4} we know that $\sigma^*$ is uniquely determined as the
center of the critical strip
(which is defined only in terms of the divisor of $f$).
Translating everything by $\sigma^*$ we can assume
that $\sigma^*=0$. By minimality the divisor of $\chi$ is
uniquely determined. Then $\chi$ is uniquely determined up to a factor
$\exp h(s)$ where $h$ is an entire
function. If $\hat \chi (s)=(\exp h(s) )\chi (s)$ gives also a functional equation for $f$, then we have
$$
f(s)=\frac{\chi (-s)}{\chi (s)} f(-s)= \frac{\chi (-s)}{\chi (s)} \frac{\hat \chi (s)}{\hat \chi (-s)} f(s).
$$
Therefore
$$
\exp(h(s)-h(-s)) =1,
$$
so for some $k\in \ZZ$,
$$
h(s)-h(-s)=2\pi i k \, .
$$
Specializing for $s=0$ we get $k=0$ and $h$ is even.

When $f$ is of  convergence exponent $d<+\infty$, and since the divisor 
of $\chi$ is contained in the divisor of $f$, then we can take $\chi$ of convergence exponent at most $d$. 
\end{proof}

If $f$ is real analytic, then it is easy to see that $\chi$ must be real analytic up to the Weierstrass
factor. We will always choose $\chi$ to be real analytic.
Then $g=\chi f$ is real analytic, and we have a four-fold symmetry and $g$ is symmetric with respect to the
vertical line $\Re s =\sigma^*$.

\begin{example}
For the Riemann zeta function $f(s)=\zeta(s)$ we have $\sigma^*=1/2$, $\sigma_-=0$, $\sigma_+=1$, 
$D_0=-2\NN^*$, $D_1= \{1 \}$, and
\begin{align*}
\chi(s) &=\pi^{-s/2}\Gamma (s/2) s(s-1), \\ 
\chi_0(s) &=\pi^{-s/2}\Gamma(s/2), \\  
R(s) &=s(s-1). 
\end{align*}
Note that 
$$
g(s)=\chi(s)\zeta(s)=\pi^{-s/2}\Gamma (s/2) s(s-1) \zeta(s)=2 \xi (s) 
$$
(using Riemann's classical notation for $\xi$).
\end{example}

Let $f$ be a meromorphic function of finite order which has its divisor contained in a left 
half plane and which has a functional equation. In order to simplify
we assume that $\sigma^*$ is not part of the divisor. For $g(s)=\chi (s) f(s)$ we have
$$
g(2\sigma^*-s)=g(s) \ ,
$$
and when we express this symmetry in the Hadamard factorization, we get
$$
Q_{g,\sigma^*} (2\sigma^*-s)=Q_{g,\sigma^*} (s) \ ,
$$
hence if we write
$$
Q_{g,\sigma^*} (s)=\sum_k a_k (s-\sigma^*)^k \ ,
$$
the symmetry implies that all odd coefficients are zero $a_1=a_3=\ldots =0$.

\begin{theorem}\label{thm:functional_eq_0structure}
 Let $f$ be a meromorphic function of exponent of convergence $d=2$, 
which has its divisor contained 
in a left half plane and has a functional equation. We assume that $\sigma^*$ is not part of the divisor.
Write $g=\chi\, f$ as before, and $\chi=\chi_0\, R$, where $R$ is a rational function symmetric with
respect to $\sigma^*$.
Then we have
$$
c_{0}(\chi_0, \sigma^*)+c_{0}(f, \sigma^*)=0.
$$
\end{theorem}
 
\begin{proof}
We observe that 
$$
Q_{g,\sigma^*}=Q_{\chi,\sigma^*} +Q_{f,\sigma^*} \ ,
$$
and for the discrepancy polynomials
$$
P_{g,\sigma^*}=P_{\chi,\sigma^*} +P_{f,\sigma^*} \ .
$$
By the above considerations, $Q_{g,\sigma^*} (s)= \sum\limits_{2k \leq g} a_{2k} (s-\sigma^*)^{2k}$.
Taking derivatives, 
 $$
 P_{g,\sigma^*}(s)=c_0 + c_1s+\ldots + c_{g-1}s^{g-1}
= - \sum\limits_{2k \leq g} 2k a_{2k} (s-\sigma^*)^{2k-1}.
 $$

If the exponent of convergence is $d=2$, then $g=1$, so the discrepancy polynomial
is constant and we have $c_0(g)=0$. Also $c_0(g)=c_0(\chi)+c_0(f)$ and
$c_0(\chi)= P_\chi =P_{\chi_0}=c_0(\chi_0)$.
Therefore we obtain the result.
\end{proof}

%

Another important property which follows from the definition of having a functional equation is the following.

\begin{proposition} \textbf{(Group property)}\label{prop:group_property}
 Dirichlet series having a functional equation with a fixed axis of symmetry form a multiplicative group.
\end{proposition}

Next, we determine when a finite Dirichlet series satisfies a functional equation.

\begin{proposition}  \label{prop:functional-eqn}
A finite Dirichlet series
$$
f(s)=1+\sum_{n=1}^N a_n e^{-\lambda_n s} \ ,
$$
satisfies a functional equation if and only if it is
of the form
 $$
 f(s)= e^{\mu s} \sum_{n=0}^{[(N-1)/2]}  a_i (e^{(-\lambda_n+\mu) s} +
  c\, e^{(\lambda_n-\mu) s})\, ,
 $$
where $c=1$ if $N$ is even, $c=\pm 1$ if $N$ is odd.
\end{proposition}

\begin{proof}
 The Dirichlet series $f(s)$ is of order $1$. Suppose that there is
 some $\chi(s)$ of order $1$ with zeros and poles in a left  cone
 such that $g(s)=\chi(s) f(s)$ is symmetric
with respect to $\Re s=\sigma^*$. By translating, we can assume $\sigma^*=0$.

The zeros of $f(s)$ lie in a strip, since $e^{-\lambda_n s} f(-s)$ is also a Dirichlet series.
Therefore $\chi(s)$ has finitely
many zeros and poles, and hence $\chi(s)=\frac{Q_1(s)}{Q_2(s)} e^{\mu s}$, for some
polynomials $Q_1(s), Q_2(s)$. The functional equation $g(s)=g(-s)$ reads
 $$
 Q_1(s)Q_2(-s) \sum_{n=0}^N a_n e^{(\mu-\lambda_n) s} =
 Q_2(s)Q_1(-s) \sum_{n=0}^N a_n e^{(\lambda_n-\mu) s} \, ,
 $$
where we have set $a_0=1$, $\lambda_0=0$.

{}From this it follows that $Q_1(s)Q_2(-s) =c\, Q_2(s)Q_1(-s)$, $c\in \CC^*$. It
follows easily that $c=\pm 1$.
Also $0,\lambda_1,\ldots, \lambda_N$ is a sequence symmetric with respect to $\mu=\lambda_N/2$.
So $\lambda_{N-i}=2\mu -\lambda_{i}$ and $a_{N-i}=a_i$.

If $N$ even, then $\lambda_{N/2}=\mu$, $c=1$,  and 
 $$
  \sum_{n=0}^N a_n e^{-\lambda_n s} = e^{\mu s} \sum_{i=0}^{N/2-1}  a_i (e^{(-\lambda_i+\mu) s} +
  e^{(\lambda_i-\mu) s}) + a_{N/2} e^{\mu s}\, .
 $$
If $N$ is odd, then 
 $$
  \sum_{n=0}^N a_n e^{-\lambda_n s} = e^{\mu s} \sum_{i=0}^{(N-1)/2}  a_i (e^{(-\lambda_i+\mu) s} +
  c \, e^{(\lambda_i-\mu) s}) \, ,
 $$
where if $c=-1$, we have $\chi(s)=s\, e^{\mu s}$.
\end{proof}

\subsection*{An example without functional equation}

Consider the elementary Dirichlet series
 \begin{equation} \label{eqn:example}
f(s)= 1 + a_1 e^{-\lambda_1 s}+a_2 e^{-\lambda_2 s}
 \end{equation}
with $0<\lambda_1<\lambda_2$ and $a_j\not=0$. It is an entire function on $\CC$ of order $1$.

If $\lambda_1,\lambda_2$ are rationally dependent, then we may write
$f(s)= 1+ a_1 \left (e^{\lambda s}\right )^{k_1} +a_2 \left (e^{\lambda s}\right )^{k_2}$, for
$\lambda_1=k_1\lambda$, $\lambda_2=k_2\lambda$, $k_1,k_2>0$ and coprime. We can compute
the zeros solving the algebraic equation $1+a_1 X^{k_1}+a_2 X^{k_2} =0$.
Therefore, the zeros of $f(s)$ lie in at most $k_2$ vertical lines, and they
form $k_2$ arithmetic sequences of the same purely imaginary step.

If $\lambda_1,\lambda_2$ are rationally independent, then we cannot compute
explicitly the zeros in general. We know that they lie in a half-plane
$\Re s<\sigma_1$. Also $a_2^{-1} e^{\lambda_2 s} f(s)$ converges
to $1$ for $\Re s\to -\infty$. So the zeros of $f(s)$ are located
in a half-plane $\Re s>\sigma_2$. Hence in a strip. By Corollary
\ref{cor:divisor}, there are infinitely many zeros in that strip.

Now, let $(\rho)$ be the set of zeros. Then $\Lambda=\{\bk=(k_{1},k_{2}) \in \NN^2 \, |
\, (k_{1},k_{2}) \neq (0,0)\}$, and
 $$
 b_\bk=\frac{(-1)^{k_1+k_2}}{k_1+k_2} \binom{k_1+k_2}{k_1} a_1^{k_1}a_2^{k_2}
 $$
and the Poisson-Newton formula on $\RR^*_+$ is
$$
 \sum_{\rho} n_\rho e^{\rho t} = \sum_{\bk}
 (\lambda_1k_1+\lambda_2k_2) b_{\bk} \delta_{\lambda_1k_1+\lambda_2k_2} \,.
$$
By Proposition \ref{prop:functional-eqn},
the Dirichlet series (\ref{eqn:example}) does not have a functional equation
unless $\lambda_2=2\lambda_1$.

\subsection*{Examples of infinite Dirichlet series with no functional equation}
Consider an infinite Dirichlet series $g$ with meromorphic extension 
to $\CC$ satisfying a functional equation 
with an infinite number of poles (taking any infinite Dirichlet series $g$ satisfying a functional equation, 
either $g$ or $g^{-1}$
has this property). For example we can take $g=1/\zeta$. Consider also the previous example 
$f(s)= 1 + a_1 e^{-\lambda_1 s}+a_2 e^{-\lambda_2 s}$, with $a_1, a_2 \not=0$, with frequencies 
$\lambda_1$ and $\lambda_2$ rationally independent with those of $g$, and $\lambda_2\not= 2\lambda_1$
so that $f$ does not have a functional equation.

Then the product $h=f\,g$ is an infinite Dirichlet series with meromorphic extension to $\CC$. 
If $h$ had a functional equation then the axes of symmetry would be the same as the one for $g$ 
(because of the symmetry of the poles), but then 
$f=h \, g^{-1}$ will have a functional equation 
from the group property \ref{prop:group_property}, which is a contradiction.  

These functions without functional equation 
do have  a Poisson-Newton formula, but in general the lack of knowledge about the location of its divisor, 
and the lack of structure of the set of frequencies makes the explicit formula of limited usefulness.

An interesting question is to determine when a classical Hurwitz zeta function has a functional equation.

\section{Explicit formulas for Riemann zeros} \label{sec:explicit-formulas}

In this section we apply our Poisson-Newton formula to the Riemann zeta function. We obtain 
a non-classical form of the Explicit Formula in analytic number theory. The classical forms can 
be derived from our distributional formula.

Explicit formulas in analytic number theory go back to the original memoir of Riemann \cite{R} on the 
analytic properties of Riemann zeta function where it is the central point of the derivation of Riemann's asymptotic 
formula for the growth of the number of primes. It relates prime numbers with non-trivial zeros of Riemann zeta function. 
Despite the mystery about the precise location of the non-trivial zeros, many of such formulas were developed at the 
end of the XIX century and the beginning of the XX century (see \cite{L}). Later, general explicit formulas
were developed by A.P. Guinand \cite{G}, J. Delsarte \cite{D}, A. Weil \cite{W} and K. Barner \cite{Ba}, 
these last ones in 
general distributional form. A classical form of this Explicit Formula is the following by K. Barner \cite{Ba}.

The Riemann zeta function is defined for $\Re s >1$ by
$$
\zeta(s)=\sum_{n\geq 1} n^{-s}=\sum_{n\geq 1} e^{-s\log n } \, ,
$$
which is a Dirichlet series with $\lambda_n =\log (n+1)$ and $\sigma_1=1$ in our notation. It
has a meromorphic extension to the complex plane $s\in \CC$ with a single simple pole at $s=1$. 
It has order $o=1$, convergence exponent $d=2$, and genus $g=1$ (see \cite{T}).

The Riemann zeta function has a single simple pole at $\rho=1$, and simple real zeros at $\rho=-2n$,
for $n=1, 2,\ldots $,
and non-real zeros in the critical strip $0<\Re s < 1$, $\rho =1/2+i\gamma$. 
The Riemann Hypothesis
conjectures that $\gamma \in \RR$, i.e., that all non-real zeros have real part $1/2$.

\begin{theorem} \label{thm:PN-zeta-Riemann}
We have
$$
\sum_{\rho} n_\rho e^{(\rho -1/2)|t|} = 2 c_0(\zeta, 1/2)\,  \delta_0 - 
\sum_{p,k\geq 1} (\log p) p^{-k/2} \left ( \delta_{k\log p} + \delta_{-k\log p} \right )  ,
$$
where
$$
c_0(\zeta , 1/2)= = -\frac{\log \pi}{2} -\frac{\pi}{4}-\frac{\gamma}{2} -\frac32 \log 2  .
$$
\end{theorem}

\begin{proof}
For $\Re s >1$ we have the Euler product which gives the relation of the zeta function with prime numbers,
$$
\zeta(s)=\prod_p (1-p^{-s})^{-1} \, ,
$$
where the product is running over the prime numbers $p$. Thus
$$
-\log \zeta (s) = -\sum_{p, \, k\geq 1} \frac{p^{-ks}}{k} =-\sum_{p, \, k\geq 1} \frac1k
e^{-k(\log p)s} \, .
$$
The vector of fundamental frequencies is $\boldsymbol{\lambda}=(\log 2, \log 3 ,\log 4 , 
\ldots )$. We have
$b_\bk=-1/k$ for $\la \boldsymbol{\lambda}, \bk\ra= k\log p$, and
$b_\bk=0$ otherwise. Therefore the stated formula follows by applying 
the Poisson-Newton formula with parameters, corollary \ref{cor:thm:symmetric_parameters}, 
for $\beta=1/2$ and $\sigma=1/2$.

For computing the value of $c_0(\zeta , 1/2)$, we use that the
Riemann zeta function has a functional equation with $\sigma^*=1/2$, $\sigma_-=0$ and $\sigma_+=1$. 
We have, using the notations of section \ref{sec:functional}, 
\begin{align*}
g(s) &=g(1-s), \\
g(s)&= \chi(s) \zeta(s), \\
\chi(s) &=\pi^{-s/2}\Gamma (s/2) s(s-1), \\ 
\chi_0(s) &=\pi^{-s/2}\Gamma(s/2), \\  
R(s) &=s(s-1). 
\end{align*}
 By theorem \ref{thm:functional_eq_0structure}, $c_0(\zeta,1/2)=-c_0(\chi_0,1/2)$. 
 The value of 
 $$
 c_0(\chi_0,0)=\frac{\log \pi}{2}+\frac{\gamma}{2}
 $$
follows from the Hadamard factorization of the $\Gamma$-function
 \begin{equation}\label{psi}
\frac{1}{\Gamma(s/2)} = \frac{s}{2} e^{\frac{\gamma}{2} s} \prod_{n\geq1} E_1(s/(-2n)) ,
 \end{equation}
thus 
\begin{equation}\label{eqn:c3}
c_0 \left (1/\Gamma(s/2) , 0\right ) = -\frac{\gamma}{2} \ .
\end{equation}
 
The zeros of $\chi_0$ are the negative integers $-2n$, $n\geq 0$, and 
are simple. Hence 
the formula (\ref{eqn:c4}) reads (for $\sigma$ not a pole of $\chi_0$)
 \begin{align*}
 -c_0(\chi_0,\sigma)+c_0(\chi_0,0) &=-\frac{1}{\sigma} + \sum_{n=1}^\infty (-1) 
\frac{-\sigma}{(-2n)(-2n-\sigma)} \\
 &= -\frac{1}{\sigma}+ \sum_{n=1}^\infty \left(\frac{1}{2n} -\frac1{2n+\sigma}\right) 
 =  \frac12 \frac{\Gamma'(\sigma /2)}{\Gamma(\sigma /2)} +\frac{\gamma}{2} ,
  \end{align*}
where the last formula follows from the expression for the logarithmic derivative of the $\Gamma$-function, 
the digamma function $\psi$,
 \begin{equation}\label{psi-psi}
\psi(s)=\frac{\Gamma' (s)}{\Gamma (s)} =-\frac1s -\gamma +\sum_{n=1}^{+\infty} \left (\frac{1}{n} -\frac{1}{n+s} \right ) ,
 \end{equation}
which results from (\ref{psi}).
 
Finally we have, for $\sigma \notin -2\ZZ$
$$
c_0(\chi_0, \sigma)=\frac{\log \pi}{2} - \frac12 \psi (\sigma/2) \, .
$$
In particular, for $\sigma=1/2$, we have (see \cite{AS}, combine entries 6.3.3, 6.3.7 and 6.3.8, p.\ 258) that
$\psi(1/4)=-\frac{\pi}{2}-3\log 2-\gamma$. Hence
$$
c_0(\zeta , 1/2)=-c_0(\chi_0, 1/2)= -\frac{\log \pi}{2} -\frac{\pi}{4}-\frac{\gamma}{2} -\frac32 \log 2 \ .
$$
\end{proof}

Note that our ``explicit formula'' (theorem \ref{thm:PN-zeta-Riemann}) is more concise than the classical formulation.
Even more if we use corollary \ref{cor:thm:symmetric_parameters} with $\beta=0$, as follows:

\begin{theorem}
We have
$$
\sum_{\rho} n_\rho e^{\rho |t|} = 2 c_0(\zeta, 0)\,  \delta_0 - \sum_{p,k\geq 1} (\log p)  
\left ( \delta_{k\log p} + \delta_{-k\log p} \right ) \ ,
$$
and 
$$
c_0(\zeta , 0) = -\log (2\pi ) \ .
$$
\end{theorem}

\begin{proof}
We can compute $c_0(\zeta , 0)$ from the known Hadamard factorization of the Riemann zeta function. We have (see \cite[p.\ 31]{T})
$$
\zeta(s) = \frac{e^{bs}}{2(s-1)\Gamma(s/2+1)} \prod_\rho  E_1(s/\rho)^{n_\rho} 
= \frac{e^{bs}}{s(s-1)\Gamma(s/2)} \prod_\rho  E_1(s/\rho)^{n_\rho}\ ,
$$
where the product is over the non-trivial zeros and $b=\log (2\pi )-1-\gamma /2$. 

Now, we have
$$
\frac{1}{s-1} =- e^{s} \left (E_1(s/1)\right )^{-1}\ ,
$$
thus
$$
c_0(\zeta , 0)=-Q_\zeta =-b  -1+c_0\left (1/\Gamma(s/2) , 0\right )  .
$$
Using (\ref{eqn:c3}) we get $c_0(\zeta, 0)=-\log ( 2 \pi )$.
\end{proof}

We can compute explicitly the contribution of the real divisor to the distribution on the left handside of Theorem 
\ref{thm:PN-zeta-Riemann},
$$
W_0(t)=-e^{|t|/2}+e^{-|t|/2}\sum_{n\geq 1} e^{-2n|t|}=-e^{|t|/2} + e^{-5|t|/2}\frac{1}{1-e^{-2|t|}}
=-e^{|t|/2} + e^{-\frac{3}{2}|t|} \frac{1}{2\sinh |t|} \, .
$$
So the associated Poisson-Newton formula on $\RR$ is
\begin{align*}
\sum_\gamma e^{ i\gamma |t|} + W_0(t) &= 2 c_0(\zeta,1/2)\,  \delta_0 - 
\sum_{p,k\geq 1} (\log p) p^{-k/2} \left ( \delta_{k\log p} + \delta_{-k\log p} \right ) ,
\end{align*}
where $\rho =1/2+i\gamma$ run over the non-trivial zeros of $\zeta$, the $p$ over prime numbers.
Following the tradition we repeat these zeros according to their multiplicity,
so we may skip the multiplicities $n_\rho =1$.

Let us see now how one can recover the classical formulation from our Poisson-Newton formula.

\begin{theorem} \label{thm:a} For a test function $\varphi$ 
such that $h(t)\varphi(t) \in \cS$, where $h$ is smooth with $h(t)= e^{-(1/2+\epsilon)|t|}$ for $|t|\geq 1$,
we have
\begin{align*}
\sum_\gamma \hat \varphi (\gamma) =  & \, \hat \varphi (i/2) + \hat \varphi (-i/2)+ 
\frac{1}{2\pi} \int_\RR \Psi(t) \hat \varphi(t) \ dt \\
&-\sum_{p,k\geq 1} (\log p) p^{-k/2} \left (\varphi (k\log p) +
\varphi (-k\log p) \right )\, ,
\end{align*}
where 
$$
\Psi(t) =-\log \pi +\Re \left ( \frac{\Gamma'}{\Gamma} (1/4+it/2)\right ) .
$$
\end{theorem}

\begin{proof}
We want to pair the distributional formula in
theorem \ref{thm:PN-zeta-Riemann} with a test function $\varphi$. By our construction, the
Poisson-Newton formula in theorem \ref{thm:PN-zeta-Riemann} is associated to
$\zeta(s-\frac12)$, which has $\sigma_1=\frac12$. Hence we take $\varphi$ such
that $h(t) \varphi(t) \in \cS$, with $h$ smooth with $h=e^{-(1/2+\epsilon)|t|}$ for $|t|\geq 1$.

Consider the Fourier transform
$$
\hat \varphi (x) =\int_\RR \varphi (t) e^{-ixt} \, dt \, .
$$
Observe that
$$
\hat \varphi (\gamma) =\int_\RR \varphi (t) e^{-i \gamma t} \ dt =\int_{\RR_+} \left (\varphi (t) e^{-i \gamma t} +
\varphi (-t) e^{-i (-\gamma ) t} \right )\, dt  \, .
$$
By the real analyticity of $\zeta(s)$, the set of non-trivial zeros is real symmetric,
$(\gamma )=(-\gamma )$, hence
$$
\sum_\gamma \hat \varphi (\gamma) =\int_{\RR_+} (\varphi (t)  +
\varphi (-t))  \left (\sum_\gamma e^{i \gamma  t}\right ) \, dt  \, .
$$

Thus applying now our Poisson-Newton formula to the test function $\varphi$ we get
$$
\sum_\gamma \hat \varphi (\gamma) + W_0[\varphi]= 2 c_0(\zeta,1/2) \varphi (0) 
-\sum_{p,k\geq 1} (\log p) p^{-k/2} (\varphi (k\log p) +
\varphi (-k\log p) )\, ,
$$
where $W_0[\varphi]$ is the functional
$$
W_0[\varphi]= \int_{\RR }  W_0 (t) \varphi (t) \, dt \, .
$$
We compute more precisely this functional. We have
\begin{align*}
W_0(t) &=-e^{-|t|/2} (W(\chi )(t)+W(\chi)(-t)) \\
&=-e^{-|t|/2} ( W(\chi_0)(t)+W(\chi_0)(-t)+W(R)(t)+W(\bar R)(-t)) \ .
\end{align*}

Note that by our assumptions, $\hat \varphi$ is holomorphic in a neighborhood 
of the strip $|\Im t |\leq 1/2$. Then we have by the general symmetric Poisson-Newton formula 
(or by direct computation)
\begin{align*}
\langle -e^{-|t|/2}  &(W(R)(t)+W(R)(-t)) , \varphi \rangle =-\int_\RR e^{-|t|/2} (e^{|t|} +1)\varphi(t) \, dt\\
&=-\int_\RR 2 \cosh (|t|/2) \varphi (t) \, dt 
=-\int_\RR 2 \cosh (t/2) \varphi (t) \, dt \\
&=-\int_\RR \varphi (t) e^{t/2}\, dt - \int_\RR \varphi (t) e^{-t/2}\, dt 
= -\hat \varphi (i/2) - \hat \varphi (-i/2)  .
\end{align*}
Now, again using 
corollary \ref{cor:symmetric-parameters-general} with $\alpha=1$ and $\beta=1/2$ 
applied to $\chi_0$ that is real analytic, we have
$$
 e^{-|t|/2} (W(\chi_0)(t)+W( \chi_0)(-t)) = 2 c_0(\chi_0 , 1/2) \delta_0 
+\cL^{-1}_{1/2}\left (2 \Re \left (\frac{\chi'_0}{\chi_0}\right ) \right ) .
$$
And using theorem \ref{thm:functional_eq_0structure},
$c_0(\chi_0 , 1/2)+c_0(\zeta , 1/2) =0$. The Poisson-Newton formula applied the test function $\varphi$ is
\begin{align*}
\sum_\gamma \hat \varphi (\gamma) =  &\, \hat \varphi (i/2) + \hat \varphi (-i/2) +\left \langle 
\cL^{-1}_{1/2}\left (2 \Re \left (\frac{\chi'_0}{\chi_0}\right ) \right ), \varphi \right \rangle \\
&-\sum_{p,k\geq 1} (\log p) p^{-k/2} (\varphi (k\log p) +
\varphi (-k\log p) )\, ,
\end{align*}

Now, we have
$$
\frac{\chi'_0 (s)}{\chi_0 (s)} =-\frac12 \log \pi +\frac12 \frac{\Gamma'(s/2)}{\Gamma(s/2)} \ ,
$$
so
$$
\left \langle 
\cL^{-1}_{1/2}\left (2 \Re \left (\frac{\chi'_0}{\chi_0}\right ) \right ), \varphi \right \rangle =
\frac{1}{2\pi} \int_\RR \Psi(t) \hat \varphi(t) \ dt \ ,
$$
where 
$$
\Psi(t) =-\log \pi +\Re \left ( \frac{\Gamma'}{\Gamma} (1/4+it/2)\right ) \ .
$$

Thus we recover the classical form of the Explicit formula given in the statement. 
\end{proof}
Historically this form is due 
to Barner that gave a new form of the Weil functional. Barner's derivation is based on an integral formula, Barner formula, that 
can be directly derived from our general Poisson-Newton formula.

\begin{remark}\label{rem:functional-equation}
  The functional equation for $\zeta$ only serves in theorem \ref{thm:PN-zeta-Riemann} to compute $c_0(\zeta,1/2)$.
Without it, we obtain theorem \ref{thm:a} except for the value of the constant in the function $\Psi(t)$.
\end{remark}

\begin{remark} \textbf{(General Explicit Formulas)}
The derivation given  
of the classical distributional Explicit Formula is general 
and applies to any Dirichlet series of order $1$ with the required conditions. In this sense the Poisson-Newton formula can be seen as 
the general Explicit Formula associated to a Dirichlet series. The structure at $0$ needs to be computed in general. 
But when we have a functional equation, one can apply the Poisson-Newton formula with the parameter 
well chosen so that the structure at $0$ vanishes (using theorem \ref{thm:functional_eq_0structure}).
The divisor on the left cone gives the general ``Weil functional'' and again, by application of the general 
Poisson-Newton formula with parameters 
we get a general Barner integral formula 
for the functional. Thus we get a general Explicit Formula with the same structure as for the classical one for the Riemann zeta function.
\end{remark}

\section{Selberg Trace formula} \label{sec:selberg}

It is well known that Selberg trace formula was developed 
by analogy with the Explicit Formulas in analytic number theory 
and that this was the original motivation by Selberg (see \cite{S}, \cite{CV}). 
In this section we explain this folklore analogy by showing that 
Selberg Trace Formula results from the Poisson-Newton formula applied to the Selberg zeta function. 
The approach is very similar 
to that of the previous section and we have a unified treatment of both formulas. 
The only relevant difference is that Selberg zeta function 
is of order $2$.

We consider a compact Riemannian surface $X$ of genus $h\geq 2$ with a metric of constant negative
curvature. Let $\cP$ be
the set of primitive geodesics. The Selberg zeta function is defined in the half plane $\Re s>1$ by
the Euler product
$$
\zeta_X(s)=\prod_{p\in \cP}\prod_{k\geq 0} \left ( 1-e^{\tau(p)(s+k)}\right ) ,
$$
where $\tau(p)$ is the length of the geodesic $p$.

We have 
\begin{align*}
-\log \zeta_X(s) &= \sum_p \sum_{k\geq 0} \sum_{l\geq 1} \frac1l \ e^{-\tau(p)(s+k) l}
=\sum_{p , l\geq 1} \frac1l\  e^{-\tau(p) ls} \ \frac{1}{1-e^{-\tau(p) l}}\\
&=\sum_{p , l\geq 1} \frac1l \ e^{\tau(p) l/2} \ \frac{1}{2 \sinh (\tau(p) l/2)}\  e^{-\tau(p) l s} \ .
\end{align*}
Thus we get the coefficients
$$
b_{p,l} =\frac1l \ e^{\tau(p) l/2} \ \frac{1}{2 \sinh (\tau(p) l/2)} \ ,
$$
and the frequencies
$$
\langle {\boldsymbol {\lambda }} , (p,l)\rangle = \lambda_{p,l} = \tau(p) l \ .
$$

One of the fundamental results of the theory is that $\zeta_X$ has a meromorphic extension 
to the complex plane of order $2$, exponent of convergence $d=3$, thus genus $g=2$ 
(see \cite{MPM-genus}), has a functional equation with $\sigma^*=1/2$, 
and its zeros are the following (see \cite[p.\ 129]{V}):

\begin{itemize}
 \item Trivial zeros at $s=-k$ with $k=0,1,2,\ldots$ with multiplicity $2(h-1)(2k+1)$.
 \item Non-trivial zeros $s=1/2\pm i \gamma_n$, $n=0,1,2,\ldots$, where $1/4+\gamma_n^2$ are the
eigenvalues of the positive Laplacian $-\Delta_X$ on $X$ counted with multiplicity. The lowest eigenvalue
$0$ yields two zeros, $s=1$ that is simple, and the trivial zero $s=0$ with multiplicity $2(h-1)$ (we 
exclude the case of  $1/4$ as eigenvalue). 
\end{itemize}

For $n<0$ write $\gamma_n=-\gamma_{-n}$. In the sense of section  \ref{sec:functional},
$\zeta_X$ satisfies a functional equation with $\sigma^*=1/2$ and $\zeta_X=\chi \, g$, where
$g(2\sigma^*-s)=g(s)$. The Newton-Cramer distribution decomposes as
$$
W(\zeta_X)=W(\chi) + W(g)\ ,
$$
where $W(\chi)$ is the contribution of the trivial zeros and $W(g)$ of the non-trivial ones. 
We compute on $\RR^*$ with $\beta =1/2$ and $\sigma=1/2$,
\begin{align*}
\widehat W(\chi ) (t) &=\sum_{n\in \ZZ}  2(h-1) (2n+1) e^{(-n-1/2)|t|} 
 \\ &=4(h-1) \sum_{n\geq 0} (n+1/2) e^{-(n+1/2) |t|} \\
&=-4(h-1)\frac{d}{d|t|} \left (\frac{1}{2\sinh (|t|/2)} \right )
=(h-1) \frac{\cosh(t/2)}{(\sinh(t/2))^2} \ , \\
\widehat W( g) (t) &=  \sum_{n\in \ZZ} e^{i\gamma_n |t|}=2 \sum_{n\geq 0}\cos(\gamma_n t) \ . 
\end{align*}

Now we apply the symmetric Poisson-Newton formula with parameter (corollary \ref{cor:thm:symmetric_parameters})
with $\beta=1/2$, and we get
$$
 \widehat W (\zeta_X)=\widehat W(\chi ) + \widehat W( g)  = 2 c_0' (\zeta_X , 1/2) \delta_0 + \sum_{p, l\in \ZZ^*} 
| \langle {\boldsymbol {\lambda }} , (p,l)\rangle |  
e^{-| \langle {\boldsymbol {\lambda }} , (p,l)\rangle |/2} \,  b_{p,|l|} \,
\delta_{\langle {\boldsymbol {\lambda }} , (p,l)\rangle} \ .
$$

To compute $c_0'(\zeta_X , 1/2)$, we use the Hadamard factorisation of $\zeta_X$. We are assuming that $s=1/2$ is not part of the divisor ($n(1/2)=0$). 
According to (\ref{eqn:psi-psi2}) with $\sigma=1/2$, we write $P_{\zeta_X,1/2}(s)=c_0+c_1s$ with
\begin{align*}
 - (\log \zeta_X)' (s)= & 
 c_0+c_1s - \sum_{n=0}^{+\infty}2(h-1)(2n+1) \frac{(s-1/2)^2}{(-n-1/2)^2} \frac{1}{s+n} \\
&-
 \sum_{n=0}^\infty \left( \frac{(s-1/2)^2}{(i\gamma_n)^2} \frac1{s-1/2-i\gamma_n} +
\frac{(s-1/2)^2}{(-i\gamma_n)^2} \frac1{s-1/2+i\gamma_n}\right).
\end{align*}
Equivalently,
\begin{align*}
 - (\log \zeta_X)' (s+1/2)= & 
 c_0+c_1(s+\frac12) +  \sum_{n=0}^{+\infty} (2-2h)(2n+1) \frac{s}{-n-1/2}\left (\frac{1}{s+n+1/2} -\frac{1}{n+1/2} \right ) \\
 &- 2s\sum_{n=0}^\infty \left( 
\frac{1}{s^2+\gamma_n^2}-\frac{1}{\gamma_n^2} \right).
\end{align*}

Now, by formulas (4.9), (4.12) and (4.17) in \cite{CV}
(formula (4.17) follows from formula (2.27) therein which can be obtained from the 
resolvent of the elliptic operator $-\Delta_X$), we have
\begin{align*}
 - (\log \zeta_X)' (s+1/2) &=(2-2h) 2s (\log \Gamma)' (s+1/2)  - 2b s- 2s\sum_{n=0}^\infty \left( 
\frac{1}{s^2+\gamma_n^2}-\frac{1}{\gamma_n^2} \right),
\end{align*}
for some $b\in\CC$.
Using (\ref{psi-psi}), we have
\begin{align*}
 (\log \Gamma)' (s+1/2)&
 =-\frac1{s+1/2} -\gamma +\sum_{n=1}^{+\infty} \left (\frac{1}{n} -\frac{1}{s+n+1/2} \right ).
\end{align*}
Multiplying by $(2-2h) 2s$ and rearranging, we get
\begin{align*}
(2-2h) 2s (\log \Gamma)' (s+1/2) =& (2-2h)4s-2s(2-2h)\gamma+
2s(2-2h)\sum_{n=1}^{+\infty} \left( \frac1n-\frac1{n+1/2}\right) \\
&+ \sum_{n=0}^{+\infty} (2-2h)(2n+1) \frac{s}{-n-1/2}\left (\frac{1}{s+n+1/2} -\frac{1}{n+1/2} \right ) .
\end{align*}
Putting everything together, this means that
\begin{align*}
c_0+c_1(s+1/2)=b's,
\end{align*}
for some $b'\in \CC$. This means that $c_1=-2c_0$ and hence $P_{\zeta_X,1/2}(s)=c_1(s-1/2)$.

Now we use Corollary  \ref{cor:thm:symmetric_parameters} applied to $\zeta_X$ with
$\alpha=1,\beta=1/2$ (in which case we have $c_1'=c_1=b'$ and $c_0'=c_0+c_1\beta=0$), obtaining
$$
 \sum_{\gamma} e^{i\gamma t} + (h-1) \frac{\cosh(t/2)}{(\sinh(t/2))^2} 
=  2 \sum_{p, l\in \ZZ^*} 
\frac{\tau(p)}{4 \sinh(\tau(p) |l|/2)} \,
\delta_{\tau(p) l} \ .
$$
This yields the classical Selberg Trace Formula as stated in \cite{CV}.

\begin{theorem} {\textbf {(Selberg Trace Formula)}} We have on $\RR$,
$$ 
\frac12 \sum_{\gamma} e^{i\gamma t} =-\frac12 (h-1) \frac{\cosh(t/2)}{(\sinh(t/2))^2} 
+ \sum_{p, l\in \ZZ^*} 
\frac{\tau(p)}{4 \sinh(\tau(p) |l|/2)} \,
\delta_{\tau(p) l} \ .
$$
\end{theorem}

\begin{remark}
We can now manipulate the integral expression for the ``Weil functional'' \`a la Barner, using the 
general Poisson-Newton formula as we have done in the previous section. These computations will be done elsewhere. 
\end{remark}

\begin{remark} \textbf{(Gutzwiller Trace Formula)}
The Selberg trace formula is just a particular case of the Gutzwiller Trace formula in Quantum Chaos (see \cite{GU}). 
The Gutzwiller Trace Formula, that is the central formula in quantum chaos, 
results from the application of the Poisson-Newton formula to the dynamical zeta function of the Dynamical System
when this zeta function has an analytic extension to the whole complex plane. 
Thus non-trivial zeros are related to the quantum energy levels and the frequencies to the classical periodic orbits. 
This will be treated elsewhere.
\end{remark}

\noindent \textbf{Acknowledgements.} We thank warmly J. Lagarias for a long list of suggestions to
a previous version of this paper. We are grateful to D. Barsky, P. Cartier and to the referee for useful comments.
First author partially supported through Spanish MICINN grant MTM2010-17389.

\end{document}